%% file: main.tex
\documentclass{article}
\usepackage[round]{natbib}
\usepackage[colorlinks=true, linkcolor=blue, citecolor=blue, urlcolor=blue,pagebackref=true]{hyperref}
\usepackage[a4paper, margin=2.5cm,top=3cm, bottom=3cm]{geometry}
\input{preamble}

\title{Maker Breaker Games on a Budget}

\author{Sebastian Lüderssen\thanks{TU Wien, Vienna, Austria}\and Fabien Nießen\thanks{KTH Royal Institute of Technology, Digital Futures, Stockholm, Sweden.}\and Silas Rathke\thanks{Fachbereich Mathematik und Informatik, Freie Universität Berlin, Arnimallee 3, 14195 Berlin, Germany.\\ \textit{Correspondence:} \href{mailto:s.rathke@fu-berlin.de}{\texttt{s.rathke@fu-berlin.de}}.}}

\begin{document}

\maketitle

\input{Chapter/0_abstract}

\smallskip

\noindent\textbf{2020 Mathematics Subject Classification.}
Primary 05C57; Secondary 91A24, 91A46.

\input{Chapter/1_introduction}

\input{Chapter/2_budget_triangle}
\input{Chapter/3_threatavoiding}
\input{Chapter/4_kfour}
\input{Chapter/5_concluding}

\section*{Declaration of AI Usage}
 During the preparation of this work the authors used ChatGPT 5.5 to assist with the proofs of \cref{thm: budget K_4,thm: kfour lower bound} about the budget $K_4$-Game.
 After using this tool, the authors reviewed and edited the content as needed and take full responsibility for the content of this paper.

\section*{Acknowledgements}
This research has been funded by the Vienna Science and Technology Fund (WWTF) [Grant ID: 10.47379/VRG23013].

This research has been funded by the Swedish Research Council project ExCLUS (2024-05603) and the Wallenberg AI, Autonomous Systems and Software Program (WASP) funded by the Knut and Alice Wallenberg Foundation.

Silas Rathke is funded by the Deutsche Forschungsgemeinschaft (DFG, German Research Foundation) under Germany's Excellence Strategy – The Berlin Mathematics
Research Center MATH+ (EXC-2046/1, EXC-2046/2, project ID: 390685689).

\bibliographystyle{plainnat}
\bibliography{bib}

\appendix
\input{Chapter/A_appendix}

\end{document}

%% file: preamble.tex
\usepackage[T1]{fontenc}
\usepackage{amssymb}
\usepackage{mathtools}
\allowdisplaybreaks
\usepackage{enumitem}
\usepackage{amsthm}
\usepackage{thm-restate}
\usepackage{xspace}
\usepackage{csquotes}
\usepackage{tikz}
\usepackage{verbatim}
\usepackage{subcaption}
\usepackage[capitalize]{cleveref}

\newtheorem{theorem}{Theorem}[section]
\newtheorem{lemma}[theorem]{Lemma}
\newtheorem{proposition}[theorem]{Proposition}
\newtheorem{claim}[theorem]{Claim}
\newtheorem{observation}[theorem]{Observation}
\newtheorem{question}[theorem]{Question}
\theoremstyle{definition}
\newtheorem{definition}[theorem]{Definition}
\newtheorem{problem}[theorem]{Problem}
\newtheorem{strategy}[theorem]{Strategy}

\newenvironment{claimproof}[1][Proof]{
  
  \begin{proof}[#1]
}{
  \end{proof}
}

\setlist[enumerate,1]{label={\textnormal{(\roman*)}}}

\usetikzlibrary{decorations.pathreplacing}

\newcommand{\floor}[1]{\left\lfloor#1\right\rfloor}
\newcommand{\ceil}[1]{\left\lceil#1\right\rceil}
\newcommand\eps{\varepsilon}
\newcommand{\coleq}{\coloneqq}
\newcommand\cF{\mathcal{F}}
\newcommand{\CH}{\mathcal{H}}
\newcommand{\HG}{H}
\newcommand{\subs}{\subseteq}
\newcommand{\abs}[1]{\left\lvert#1\right\rvert}
\newcommand{\erdos}{Erd\H{o}s\xspace}
\newcommand{\Nearly}{N^+}
\newcommand{\Ninvearly}{N^-}
\newcommand{\degM}{\deg_M}
\newcommand{\degML}{\deg^L_M}
\newcommand{\degB}{\deg_B}
\newcommand{\bG}{\mathbb G}
\newcommand{\norm}[1]{\left\lVert #1 \right\rVert}
\renewcommand{\O}{\mathcal O}
\DeclareMathOperator{\threats}{th}
\newcommand{\nH}{h}
\newcommand{\nT}{t}
\newcommand{\nD}{g}
\newcommand{\nC}{c}
\newcommand{\nS}{s}
\newcommand{\bud}{b}
\DeclareMathOperator{\budget}{bud}

\newcommand{\mbg}{$(1:q)$ Maker Breaker Game\xspace}
\newcommand{\tg}{$(1:q)$ Triangle Game\xspace}
\newcommand{\btg}{$(1:q)$ Budget Triangle Game\xspace}
\newcommand{\abtg}{$(1:q)$ $a$-Budget Triangle Game\xspace}
\newcommand{\btatg}{$(1:q)$ Budget Threat-Avoiding Triangle Game\xspace}
\newcommand{\bkfourg}{$(1:q)$ Budget~$K_4$-Game\xspace}
\newcommand{\kfourg}{$(1:q)$ $K_4$-Game\xspace}
\newcommand{\qkfour}{q_{K_4}}
\newcommand{\qTG}{q_\triangle}
\newcommand{\hgame}{$(1:q)$ $H$-Game\xspace}
\newcommand{\bhgame}{$(1:q)$ Budget $H$-Game\xspace}
\newcommand{\hgames}{$(1:q)$ $H$-Games\xspace}
\newcommand{\hypermakerbreaker}{Hypergraph Maker Breaker Game\xspace}
\newcommand{\hypermakerbreakers}{Hypergraph Maker Breaker Games\xspace}

%% file: Chapter/0_abstract.tex
\begin{abstract}
    The Maker Breaker Triangle Game involves two players, Maker and Breaker, who alternately claim 1 and $q$ edges of $K_n$, respectively. Maker's goal is to claim all three edges of any triangle, whereas Breaker's goal is to prevent this. The threshold bias, i.e. the minimum $q$ such that Breaker wins, is known to lie between $\big\lceil\sqrt{2n-2}-\frac{3}{2}\big\rceil$ and $\big(\sqrt{8/3}+o(1)\big)\sqrt{n}$. Determining its exact value is a longstanding open problem.
    
    In this paper, we introduce a novel version of the game in which Breaker may claim fewer than $q$ edges per round to build up a budget that he can spend in later rounds.
    With this additional power for Breaker, we determine the threshold bias to be precisely $\big\lceil\sqrt{2n-2}-\frac{3}{2}\big\rceil$ for all~$n$, matching the known lower bound. This is the first version of the Maker Breaker Triangle Game for which the exact threshold bias is known. Even if Breaker is not allowed to use the budget for threats, we prove that the threshold bias is still $\big(\sqrt{2}+o(1)\big)\sqrt{n}$.
    
    Furthermore, we study the budget version of other Maker Breaker Games. Specifically, for the $K_4$-Game, we prove the first explicit lower and upper bounds on the threshold bias in both the original and the budget version of the game.
\end{abstract}

%% file: Chapter/1_introduction.tex
\section{Introduction}\label{sec: introduction}
\cite{chvatal1978biased} introduced the \emph{biased \mbg} on the complete graph $K_n$ for a fixed positive integer $n$ and a family of \emph{winning sets} $\cF\subs 2^{E(K_n)}$. It is played by two players, Maker and Breaker, in the following way: in each round, Maker first claims one of the previously unclaimed edges of $K_n$, whereafter Breaker claims $q$ of the previously unclaimed edges of $K_n$ (unless there are fewer than $q$ unclaimed edges, in which case he claims all remaining edges). Maker wins the game if, at some point, she has claimed all elements of a winning set $F\in \mathcal F$. Breaker wins if he has claimed one element of each winning set. In each Maker Breaker game, either Maker or Breaker has a winning strategy. For background on Maker Breaker Games and other positional games, we refer to \cite{krivelevich2014positional}.
\cite{chvatal1978biased} asked for the \emph{threshold bias}, i.e., given $n$ and $\cF\subs 2^{E(K_n)}$, the smallest $q$ such that Breaker has a winning strategy. They gave upper and lower bounds for the threshold bias for several different games of this kind.

One of the most famous examples is the \emph{\tg} where the winning sets $\cF$ are all the triangles of $K_n$. Let $\qTG(n)$ be the smallest $q$ such that Breaker has a winning strategy in the \tg. Chvátal and \erdos{} showed\footnote{The lower bound proven by \cite{chvatal1978biased} is actually $\ceil{\sqrt{2n+2}-\frac{5}{2}}$ which comes from the fact that there, Breaker starts the game instead of Maker.}
\begin{equation}\label{eq: chvatal erdos triangle bounds}
    \ceil{\sqrt{2n-2}-\frac{3}{2}}\le \qTG(n)\le \ceil{2\sqrt n}.
\end{equation}

The upper bound was improved by \cite{triangleimprovement} to 
\[\qTG(n)\le(2-1/24)\sqrt n\] if $n$ is large enough. Finally, \cite{glazik2022New} improved the upper bound even further in 2022, showing that 
\[\qTG(n)\le \sqrt{\left(\frac{8}{3}-o(1)\right)n}\approx 1.633\sqrt n.\] 

In the last nearly 50 years, the lower bound has not been improved, which led to the belief that the lower bound is essentially sharp, i.e.\@ that $\qTG(n)=\big(\sqrt 2+o(1)\big)\sqrt n$ (e.g. \cite{triangleimprovement}).\medskip

Since this has proven to be hard to show, it is natural to look at closely related variants of the Triangle Game with the aim of proving sharp bounds \citep{randomboard,randomboard2,dean2016client,beck2002positional}. This provides insight into how to tackle the original \tg. In particular, we ask the following question:

\begin{question}
How much additional power do we need to give Breaker so that we can prove that the lower bound in \Cref{eq: chvatal erdos triangle bounds} is essentially optimal?
\end{question}

Any Breaker strategy that tightens the threshold bias towards the lower bound must, in particular, defeat the construction behind the lower bound. In the strategy of \cite{chvatal1978biased}, Maker chooses a vertex $r$ and repeatedly claims edges incident to it, building a star. For the stated range of $q$, when she can no longer extend the star, two of its leaves still form an unclaimed edge. She then claims that edge to win the game.

After Maker's first edge $\{u,v\}$, Breaker does not know which endpoint will become the centre. In this particular strategy, Maker's second edge reveals her choice. The uncertainty can, however, last longer: Maker may first claim several disjoint edges and then join vertices of equal Maker degree in different components, retaining two possible centres. Breaker can therefore face repeated decisions about where to claim his edges before Maker commits to one centre.

This motivates giving Breaker the power to \enquote{wait} with claiming his edges, i.e. to build a budget of moves, which he can use in later turns. This may help Breaker avoid committing to edges too early and instead allocate them according to how the game develops. We now formalise this budget variant of the game.

\begin{definition}
For a fixed positive integer $n$, a family $\mathcal F\subs 2^{E(K_n)}$ of \emph{winning sets} and a non-negative integer $q$, called the \emph{bias}, the \emph{$(1:q)$ Budget Maker Breaker Game} is the following game with two players, called Maker and Breaker. The players take alternate turns, with Maker beginning. In her turn, Maker claims one of the previously unclaimed edges of $K_n$. In his turn, Breaker can claim an arbitrary number of previously unclaimed edges of $K_n$ such that at the end of his turn, the number of edges claimed by Breaker is at most $q$ times the number of edges claimed by Maker. 

Maker wins if she manages to claim all elements of an $F\in\mathcal F$. Breaker wins if he manages to claim an element of each winning set. 

If the winning sets are all the triangles of $K_n$, we call it the \textit{\btg} and denote its threshold bias by $\qTG^b(n)$.
\end{definition}

In contrast to the original Maker Breaker Game, in the Budget Maker Breaker Game, Breaker can choose not to claim all $q$ edges immediately, but instead he can build up a budget of edges that he can use later. As this gives Breaker more options to play, we immediately get \[\qTG^b(n)\le \qTG(n).\]

As it turns out, having a budget is enough power for Breaker so that we can determine the threshold bias $\qTG^b(n)$ precisely and it is exactly the lower bound of \eqref{eq: chvatal erdos triangle bounds}.

\begin{restatable}{theorem}{thmbudgettriangle}\label{thm: budget triangle}
    For every integer $n\ge 3$, the threshold bias of the \btg is exactly $\qTG^b(n)=\ceil{\sqrt{2n-2}-\frac{3}{2}}$.
\end{restatable}

With a complete understanding of the \btg{}, a canonical next step is to take away some of the new power of Breaker and see if we can still prove the same threshold bias. One crucial restriction of Breaker in the original \tg is that he can claim at most $q$ threats. Here, a \textit{threat} is an edge such that, if Maker claimed it, she would win immediately. In other words, it is an unclaimed edge that, together with two edges claimed by Maker, forms a triangle. Therefore, Maker wins the original Triangle Game if she can create $q+1$ threats in a single turn. In the Budget Triangle Game, this is no longer true as Breaker might build up enough budget to deal with more than $q$ threats in a single turn.

Thus, a natural game \enquote{between} the original Triangle Game and the Budget Triangle Game is a game with a budget, but with the restriction that Breaker is only allowed to claim at most $q$ threats per turn:

\begin{definition}
    In the \emph{$(1:q)$ Budget Threat-Avoiding Triangle Game}, two players, Maker and Breaker, take alternate turns with Maker beginning. In her turn, Maker claims one of the previously unclaimed edges of $K_n$. In his turn, Breaker can claim at most $q$ threat edges and an arbitrary number of previously unclaimed edges of $K_n$ such that at the end of his turn, the number of edges claimed by Breaker is at most $q$ times the number of edges claimed by Maker.

    Maker wins if she manages to claim all edges of a triangle. Otherwise, Breaker wins. We denote its threshold bias by $\qTG^t(n)$.
\end{definition}
Note that we have \[\qTG^b(n)\le \qTG^t(n)\le \qTG(n).\]

For this game, we are also able to show that the lower bound of \Cref{eq: chvatal erdos triangle bounds} is essentially tight. 

\begin{restatable}{theorem}{threatavoiding}
\label{thm: threatavoiding}
    For every integer $n\ge 3$, the threshold bias of the \btatg fulfils $\qTG^t(n) \in \left[\left\lceil\sqrt{2n-2}-\frac{3}{2}\right\rceil,\left\lceil \sqrt{2n}\right\rceil+1\right]$. In particular, $\qTG^t(n) = \big(\sqrt{2}+o(1)\big)\sqrt{n}$.
\end{restatable}
Note that \Cref{thm: threatavoiding} determines the threshold bias up to an additive error of 3 for all $n \ge 3$.

We discuss other ways to restrict Breaker's power in \Cref{sec: Concluding}, along with future research questions.\medskip 

Next, we turn our attention to more general settings, in which Maker's winning sets do not derive from triangles but from a fixed graph~$H$. More formally, the \textit{\hgame} is the \mbg in which the winning sets are the edge sets of all subgraphs of $K_n$ isomorphic to $H$. Again, the threshold bias $q_H(n)$ is defined as the smallest $q$ so that Breaker has a winning strategy. \cite{bednarska2000biased} show that if $H$ has at least two edges, the threshold bias is $q_H = \Theta\big(n^{1/m(H)}\big)$ where \[m(H)\coleq\max_{\substack{H'\subs H\\ v(H')\ge 3}}\frac{e(H')-1}{v(H')-2}.\] 

Bounds on the constant factor of $q_H(n)$ are only known for specific graphs~$H$ besides the triangle. \cite{MinorGame} determines\footnote{His result is phrased for the game where Maker only needs to build a graph that has $H$ as a minor, but if $H$ is one of the graphs stated above, then having $H$ as a minor is equivalent to having $H$ as a subgraph.} the leading constant if $H$ is a matching, a forest whose components have at most two edges, a path, or a three-legged spider. Furthermore, he shows \[\big(1-o(1)\big)\frac{n}{r-1}\le q_{K_{1,r}}(n)\le \big(2+o(1)\big)\frac{n}{r-1}.\]
Bounds by \cite{sowa2025constructive, sowa2026constructive} show that 
\[\big(0.16-o(1)\big)n^{2/3}\le q_{C_4}(n)\le \big((27/4)^{1/3}+o(1)\big)n^{2/3},\]
and, more generally, for a fixed integer $k\ge 3$ and sufficiently large $n$,
\[q_{C_k}(n)\le \floor{\sqrt[k-1]{(k-1)\left(\frac{2(k-1)}{k}\right)^{k-2}n^{k-2}}}+1.\]
Finally, \cite{C6Paper} showed that for sufficiently large $n$,
\[0.007n^{4/5}\le q_{C_6}(n).\]
To our knowledge, these are all graphs where bounds on the constant factor are known. Motivated by the results for the \tg, one can ask the following question:

\begin{question}
Can we obtain tighter bounds on the threshold bias in the \hgame when allowing Breaker to use a budget?
\end{question}
To this end, we define the \textit{\bhgame} as the version of the \hgame in which Breaker is allowed to use a budget and $q_H^b(n)$ as the corresponding threshold bias. Proving bounds on $q_H^b(n)$ likely provides insight on how to determine $q_H(n)$. For this, we should first make sure that $q_H^b(n)$ and $q_H(n)$ have the same order of magnitude. And indeed, we can show that this is the case:
\begin{restatable}{proposition}{Hresult}\label{prop: H-game thresholds}
    For every fixed graph $H$ with at least two edges, the threshold bias of the \bhgame fulfils $q_{H}^b(n)=\Theta\big(n^{1/m(H)}\big)$.
\end{restatable}
Hence, studying the \bhgame might help in solving the \hgame. In the case of $H = K_4$, we are able to give an explicit upper bound for the budget version.

\begin{restatable}{theorem} {thmbudgetkfour}\label{thm: budget K_4}
    The threshold bias of the \bkfourg fulfils the upper bound
    \[q_{K_4}^b(n) \leq \left(5\cdot 2^{-9/5}e^{2/5} + o(1) \right)n^{2/5}\approx 2.14\,n^{2/5}.\] 
\end{restatable}

Using the same techniques, we can prove the first explicit upper bound in the original version. 

\begin{restatable}{theorem}{thmnobudgetkfour}\label{thm: non-budget K_4}
    The threshold bias of the (original) \kfourg fulfils the upper bound
    \[q_{K_4}(n) \leq \left(5^{6/5}\;2^{-9/5}e^{2/5} + o(1) \right)n^{2/5}\approx 2.96\,n^{2/5}.\] 
\end{restatable}

We complement this with the first explicit lower bound on the threshold bias of the \kfourg. 

\begin{restatable}{theorem}{kfourlowerbound}\label{thm: kfour lower bound}
    For $n$ large enough, the threshold bias of the \kfourg and the \bkfourg fulfils the lower bound $\qkfour(n) \geq 6\cdot 10^{-6}n^{2/5}$.
\end{restatable}

\paragraph{Related work.}
Several variants of Maker Breaker Games modify Breaker's power while keeping the general setup of the game. 
\cite{pegden2025structure} study \emph{structure-biased} games, in which Breaker's edges in each turn are required to form a prescribed structure, such as a matching, a clique, or a star. 
In particular, they determine the order of magnitude of the threshold bias for several structure-biased versions of the Triangle Game.

\cite{krivelevich2015random} consider, among other variants, a random Breaker who claims his $q$ edges uniformly at random in every turn, while Maker still plays strategically. 
Conversely, \cite{clemens2025maker} study a variant where Breaker is given more power through an informational advantage. They consider games with a \emph{phantom} breaker: Breaker has complete information, whereas his claimed edges are hidden from Maker until she attempts to claim one of them.

Another related modification is the Client Waiter Game, introduced by \cite{beck2002positional} to better understand Maker Breaker Games. 
Here, the Breaker-like player Waiter offers a collection of unclaimed edges, Client chooses one of them, and Waiter receives the remaining edges.
Biased and random-graph versions of Client Waiter Games were studied by \cite{dean2016client}.

Instead of modifying the players' moves, another line of work changes the underlying board.
\cite{randomboard} and \cite{randomboard2} study clique games on the random graph $G_{n,p}$.
More recently, \cite{clemens} initiated the study of minimal boards on which Maker has a winning strategy in the $H$-Game, and \cite{brinkmann2026faster} presented algorithms for deciding the unbiased ($q=1$) Triangle Game on general graphs.
\cite{hgameforhypergraphs} generalise the threshold-bias results of \cite{bednarska2000biased} from graph $H$-Games to a broad class of hypergraph Maker Breaker games.

Regarding the $K_\ell$-Game, \cite{beckbook,1to1clique} determines asymptotically the largest clique Maker can build in the unbiased game, while \cite{gebauer} proves a lower bound in the biased variant.

\begin{paragraph}{Notation.}
    Throughout the paper, we denote by $G_M$ resp.\@ $G_B$ the graph on all $n$ vertices containing exactly the edges claimed by Maker resp.\@ Breaker. For a vertex $v$, $\degM(v)$ resp.\@ $\degB(v)$ denote the degree of $v$ in $G_M$ resp.\@ $G_B$. Denote by $e(G_M)$ and $e(G_B)$ the number of edges in $G_M$ and $G_B$, respectively. 
\end{paragraph}

\begin{paragraph}{Outline of the paper.}
    In \Cref{sec: budget triangle game}, we prove \Cref{thm: budget triangle}, which determines the exact threshold bias for the \btg.
    In \Cref{sec: threat-avoiding}, we analyse the threat-avoiding version of the \tg and prove \Cref{thm: threatavoiding}.
    \Cref{thm: budget K_4,thm: non-budget K_4}, which give upper bounds on the threshold bias of the \kfourg with and without a budget, respectively, are proved in \Cref{sec: K_4 game}.
    The lower bound for the \kfourg is proved in Appendix~\ref{sec: appendix lower bound kfourg}.
    Furthermore, \Cref{prop: H-game thresholds} is proved in Appendix~\ref{sec: appendix asymp bound}.
    In the concluding section, \cref{sec: Concluding}, we state some open problems and possible directions for further research.
\end{paragraph}

%% file: Chapter/2_budget_triangle.tex
\section{The \texorpdfstring{\btg}{}}\label{sec: budget triangle game}

In this section, we study the \btg and prove \Cref{thm: budget triangle} which determines $\qTG^b(n)$ precisely. For completeness, we restate the theorem here.

\thmbudgettriangle*

We start with the lower bound.

\begin{restatable}{lemma}{tglowerbound}\label{lem: tg lower bound}
For every integer $n\ge 3$, we have $\qTG^b(n)\ge\ceil{\sqrt{2n-2}-\frac{3}{2}}$.
\end{restatable}

To prove \Cref{lem: tg lower bound}, we have to find a winning strategy for Maker if $q<\sqrt{2n-2}-\frac{3}{2}$. It turns out that the same strategy works as the one described by \cite{chvatal1978biased} when they initiated the study of the \tg. There, Maker builds a star with arbitrary centre $r$ and it can be shown that if $q< \sqrt{2n-2}-\frac{3}{2}$, once Maker can no longer add a new edge to $r$, there are two edges $\{u,r\}$ and $\{w,r\}$ in $G_M$ such that the edge $\{u,w\}$ is still unclaimed, as otherwise the number of edges Breaker claimed is larger than $q\cdot e(G_M)$. Since $q\cdot e(G_M)\ge e(G_B)$ is also a constraint in the \btg, the very same analysis also works here to prove \Cref{lem: tg lower bound}. We do not repeat the full argument here, but for completeness, the proof of \Cref{lem: tg lower bound} is written up in \Cref{sec: lower bound}.

\medskip

Next, we consider the upper bound on $\qTG^b(n)$, for which we have to describe Breaker's winning strategy for the case $q\ge \sqrt{2n-2}-\frac{3}{2}$.
Recall that an edge $\{u,v\}$ is called a \emph{threat} if it is unclaimed and there is a vertex $w$ such that $\{u,w\}\in E(G_M)$ and $\{w,v\}\in E(G_M)$. If Breaker does not claim all newly created threat edges in his next turn, then Maker can win in the next round.

The idea behind the strategy of Breaker comes from the following observation. If Maker claims the edge $\{u,v\}$, then all newly created threat edges are either incident to $u$ or~$v$. More precisely, if $w$ is a neighbour of $u$ in $G_M$, then $\{w,v\}$ becomes a threat unless it is already claimed. Similarly, $\{w,u\}$ becomes a threat if $w$ is a neighbour of $v$ in $G_M$ unless it is already claimed. Therefore, Maker creates approximately $\degM(u)+\degM(v)$ threats when claiming $\{u,v\}$. Thus, if $u$ or $v$ has a large Maker degree, a lot of threats could be created. 
This is why, in our strategy, Breaker not only closes all threats in his turn but also makes sure that no vertex will ever have a Maker degree larger than some threshold $\theta$. As soon as a vertex $v$ reaches a Maker degree of $\theta$, Breaker claims all edges incident to $v$. 

It remains to determine the value of $\theta$. In the analysis of Breaker's strategy, we will no longer consider a total budget. Instead, each vertex will have its own budget, and Breaker edges are \enquote{paid for} by their endpoints. We will see that as soon as a vertex has a Maker degree of $q+1$, it could pay for all incident edges by itself; hence, we set $\theta$ to $q+1$. Now, we are ready to formulate the entire strategy. 

\begin{strategy}\label{strat: Breaker triangle}
    At each turn, Breaker does the following:
    \begin{enumerate}
        \item\label{enum: threat} Claim all threat edges.
        \item\label{enum: close} For every vertex $v$, if $\degM(v)=q+1$, claim all unclaimed edges incident to $v$.
    \end{enumerate}
     We call the process of claiming all edges incident to $v$ in \ref{enum: close} \emph{closing} $v$. Vertices for which this has been done are called \emph{closed}; all others are called \emph{open}.
\end{strategy}

Before we prove that this strategy really is a winning strategy for $q\ge\sqrt{2n-2}-\frac{3}{2}$, we first show that this desired bound is equal to another expression, which will be the expression we get from the analysis of the strategy.

\begin{lemma}\label{lem: ceil eq}
    For all positive integers $n$, we have 
    \[\ceil{\sqrt{2n-2}-\frac{3}{2}}=\ceil{\sqrt{2n-\frac{7}{4}}-\frac{3}{2}}.\]
\end{lemma}
\begin{proof}
    Suppose there is a positive integer $n$ such that $\ceil{\sqrt{2n-2}-\frac{3}{2}}<\ceil{\sqrt{2n-\frac{7}{4}}-\frac{3}{2}}$. That means that there is an integer $k$ such that
    \begin{align*}
    &&\sqrt{2n-2}-\frac{3}{2}\le k&<\sqrt{2n-\frac{7}{4}}-\frac{3}{2}\\
    &\Longrightarrow&2n-2\le\left(k+\frac{3}{2}\right)^2&<2n-\frac{7}{4}\\
    &\Longrightarrow&8n-8\le\left(2k+3\right)^2&<8n-7.
    \end{align*}
    In the last line, all terms are integers. Therefore, we get $8n-8=\left(2k+3\right)^2$, which is impossible as the left-hand side is even, whereas the right-hand side is odd. 
\end{proof}

Now, we are ready to finish the proof of \Cref{thm: budget triangle}.

\begin{proof}[Proof of \Cref{thm: budget triangle}.]
    From \Cref{lem: tg lower bound}, we already know that $        \qTG^b(n)\ge\ceil{\sqrt{2n-2}-\frac{3}{2}}$.
    Therefore it remains to prove that for every integer $n\ge 3$, if $q$ is an integer with $q\ge \sqrt{2n-2}-\frac{3}{2}$, Breaker wins the \btg. Indeed, we prove that Breaker wins using \Cref{strat: Breaker triangle}.
    
Fix an integer $n\ge 3$ and let $q$ be an integer with $q\ge \sqrt{2n-2}-\frac{3}{2}$.
By \ref{enum: threat} of \Cref{strat: Breaker triangle}, Maker will never claim a triangle. All we have to show is that Breaker can always perform the steps \ref{enum: threat} and \ref{enum: close} of \Cref{strat: Breaker triangle} without claiming more than $q\cdot  e(G_M)$ edges. Because of \ref{enum: close}, we can immediately conclude
\begin{equation}\label{eq: max M deg}
    \degM(v)\le q+1\hspace{1cm}\forall~ v\in V(G)
\end{equation}
at every time during the game.

The idea of the analysis is the following. Instead of having a single total budget for Breaker, we divide his budget among the vertices; more precisely, every claimed Maker edge distributes $q/2$ budget to both its endpoints. Furthermore, edges claimed by Breaker because of threats are then paid for by their endpoints with $1/2$ budget each, whereas edges claimed while closing a vertex are fully paid for by the vertex that is being closed. Our goal is to prove that no vertex spends more budget than it was assigned. 

Following this idea, we maintain a budget function $\bud\colon V(G)\to \mathbb R$, defined as
\[\bud(v)=\begin{cases}
    \frac{q}{2}\cdot \degM(v)-\frac{1}{2}\degB'(v)&\text{$v$ open}\\0&\text{$v$ closed,}
\end{cases}\]
where $\degB'(v)$ is the number of edges incident to $v$ that were added because of \ref{enum: threat}, i.e., because they were threat edges.

\begin{claim}\label{cl: c(v) nonnegative}
    Throughout the game, $\bud(v)\ge 0$ for all $v\in V(G)$.
\end{claim}
\begin{claimproof}
    Fix any vertex $v\in V(G)$ and any time of the game. If $v$ is closed, then there is nothing to show. Let $E'$ be the set of edges incident to $v$ that were added because they were threat edges. In particular, $\abs{E'}=\degB'(v)$. Each $\{u,v\}\in E'$ was added because there was a vertex $w$ such that $\{u,w\},\{w,v\}\in E(G_M)$. By \eqref{eq: max M deg}, for every $w\in N_{G_M}(v)$, there can be at most $q$ many nodes $u\ne v$ such that $\{u,w\}\in E(G_M)$. Hence, $\degB'(v)\le q\cdot\degM(v)$. Therefore, $\bud(v)\ge 0$.
\end{claimproof}
\begin{claim}\label{cl: claim 2}
    Throughout the game, we always have \begin{equation}\label{eq: claim 2}\sum_{v\in V(G)}\bud(v)+e(G_B)\le q\cdot e(G_M).
    \end{equation}
\end{claim}
Once we have proven the claim, we are done because then Claim~\ref{cl: claim 2} together with \Cref{cl: c(v) nonnegative} immediately implies $e(G_B)\le q\cdot e(G_M)$, whence Breaker can always perform his strategy without claiming more edges than his budget allows.
\begin{claimproof}[Proof of Claim~\ref{cl: claim 2}]
    The statement holds at the beginning of the game. Suppose \eqref{eq: claim 2} holds and Maker claims an edge $\{u,v\}$. In particular, this means that $u$ and $v$ are open. Therefore, both sides of \eqref{eq: claim 2} increase by exactly $q$, whence it still holds after Maker's turn.

    Now suppose \eqref{eq: claim 2} holds and Breaker takes a turn. When he is claiming a threat edge $\{u,v\}$, it means that $u$ and $v$ are still open. Therefore, $e(G_B)$ increases by 1 whereas $\bud(u)$ and $\bud(v)$ both decrease by $\frac{1}{2}$. Thus, both sides of \eqref{eq: claim 2} do not change and \eqref{eq: claim 2} still holds after Breaker claims $\{u,v\}$. 

    Finally, we have to consider the case that \eqref{eq: claim 2} holds and Breaker closes a vertex $v$. That means that he claims $n-1-\degM(v)-\degB(v)=n-1-(q+1)-\degB(v)\le n-1-(q+1)-\frac{1}{2}\degB'(v)$ edges in total, where the inequality uses $\degB(v)\ge \degB'(v)\ge \frac{1}{2}\degB'(v)$. Thus, $e(G_B)$ increases by at most $n-q-2-\frac{1}{2}\degB'(v)$. Meanwhile, $\bud(v)$ decreases by $\frac{q}{2}\cdot\degM(v)-\frac{1}{2}\degB'(v)=\frac{q^2+q}{2}-\frac{1}{2}\degB'(v)$. To show that \eqref{eq: claim 2} still holds, it is enough to show 
    \begin{align*}
        &&n-q-2-\frac{1}{2}\degB'(v)&\le \frac{q^2+q}{2}-\frac{1}{2}\degB'(v)\\
        &\!\iff&0&\le q^2+3q-(2n-4)\\
        &\Longleftarrow&q&\ge -\frac{3}{2}+\sqrt{\frac{9}{4}+2n-4}=-\frac{3}{2}+\sqrt{2n-\frac{7}{4}}.  
    \end{align*}
    As we picked $q$ such that $q\ge \ceil{\sqrt{2n-2}-\frac{3}{2}}=\ceil{\sqrt{2n-\frac74}-\frac{3}{2}}$, where the equality follows from \Cref{lem: ceil eq}, the inequality above is satisfied by the definition of $q$, which concludes the proof of the claim.
    \end{claimproof}
    Therefore $\qTG^b \leq \ceil{\sqrt{2n-2}-\frac{3}{2}}$, which concludes the proof.
\end{proof}

The overall goal is to determine the threshold function of the original non-budget Triangle Game. Therefore, one might ask how much we actually used the additional power of having a budget in \Cref{strat: Breaker triangle}. For this, we need a way to measure the budget usage of a winning strategy in the \btg.

\begin{definition}\label{def: budget usage}
    Consider a concrete play of the \btg and let $b\coleq (b_1,\dots,b_k)$ be the \emph{move sequence} where $b_i$ is the number of edges Breaker claimed in round $i$. Let $b'\coleq (b_1',\dots,b_k')$ be the \emph{excess sequence} where $b_i'\coleq \max(b_i-q,0)$. The \emph{budget usage} is then defined as 
    \[\budget(b)\coleq\sum_{i=1}^kb_i'.\]

    If $\mathcal S$ is a winning strategy for Breaker in the \btg, then its \emph{budget usage} $\budget(\mathcal S)$ is defined as the smallest integer $\ell$ such that, whenever Breaker plays according to strategy $\mathcal S$, his play will have a budget usage of at most~$\ell$.  
\end{definition}

For our \Cref{strat: Breaker triangle}, we can show that its budget usage has, asymptotically, the worst possible budget usage. 

\begin{restatable}{proposition}{propbudgetusage}\label{prop: budget usage}
     Let $q = \qTG^b(n)$ and let $\mathcal S$ be \Cref{strat: Breaker triangle} for the \btg played on $K_n$. Then $\budget(\mathcal S)=\Theta\big(n^2\big)$.
\end{restatable}
\begin{proof}
    As Breaker claims at most $\binom{n}{2}$ edges, we have $\budget(\mathcal S)\le \binom{n}{2}=\O\big(n^2\big)$.

    For the other direction, we give a strategy for Maker such that, if Breaker plays according to \Cref{strat: Breaker triangle}, then his play has a budget usage of $\Omega\big(n^2\big)$. For this, let $k\coleq \floor{\frac{n}{4}}$ and take $2k$ distinct vertices $u_1,\dots,u_k,w_1,\dots,w_k$. Note that $q < k$. Let $G$ be the bipartite graph with vertex set $\{u_1,\dots,u_k,w_1,\dots,w_k\}$ in which $u_i$ is connected to $w_i,w_{i+1},\dots,w_{i+q-1}$ where the indices are considered mod~$k$. One can quickly see that $G$ is $q$-regular.
    
    Maker's strategy is to successively claim the edges of $G$. As no vertex reaches a degree of $q+1$, Breaker will not close a vertex while Maker claims the edges of~$G$. Breaker will, however, close threats, but as $G$ is bipartite, it is triangle-free, so Breaker will never claim an edge of $G$. Therefore, Maker can claim all edges of~$G$.

    Then, all $u_1,\dots,u_k,w_1,\dots,w_k$ have a Maker degree of $q$ and all other vertices are still isolated in both $G_M$ and $G_B$. In the next $\le 2k$ rounds, while there is still a vertex $v\in \{u_1,\dots,u_k,w_1,\dots,w_k\}$ which is incident to an unclaimed edge, Maker will claim an unclaimed edge incident to~$v$. This forces Breaker to close~$v$. Thus, after the next $\le 2k$ rounds, all vertices in $\{u_1,\dots,u_k,w_1,\dots,w_k\}$ are closed, so Breaker claimed at least 
    \[2k\cdot (n-2k)-2k=\Omega\big(n^2\big)\]
    edges in $2k$ rounds. Since \Cref{thm: budget triangle} implies $2k\cdot q=\O\big(n^{3/2}\big)$, also the sum of the $2k$ corresponding entries of the excess sequence is $\Omega\big(n^2\big)$, so $\budget(\mathcal S)=\Omega\big(n^2\big)$. 
\end{proof}

It is an open question whether there are winning strategies for Breaker with which he can win the $(1:q)$ Budget Triangle Game with $q=\qTG^b(n)$ whose budget usage is $o(n^2)$. We discuss this further in the concluding remarks in \Cref{sec: Concluding}.

%% file: Chapter/3_threatavoiding.tex
\section{The \texorpdfstring{\btatg}{}}\label{sec: threat-avoiding}

In this section, we consider the \btatg, a special variant of the \btg in which Breaker can only claim at most $q$ threat edges per turn. In particular, Maker wins not only if she claims a triangle but also if she creates more than $q$ threats in one turn.
In the original \tg, creating $q+1$ threats guarantees a Maker win in the next move, while in the budget version, Breaker can use his budget to close the extra threats.
Thus, solving the \btatg is a natural step on the path of converting a strategy for the \btg into a non-budget strategy for the original \tg.

The goal of this section is to prove the following:

\threatavoiding*

The lower bound of $\qTG^t(n)$ follows directly from \Cref{lem: tg lower bound}. Therefore, we only need to show the upper bound.

For the remainder of the section, let $q$ be the smallest \emph{even} integer satisfying $q \geq \left\lceil\sqrt{2n}\right\rceil$. Note that $q\le \left\lceil\sqrt{2n}\right\rceil+1$. The assumption that $q$ is even makes some of the calculations in the proof easier. We believe that a more thorough analysis can avoid this assumption. In that case, the $+1$ in the upper bound of \Cref{thm: threatavoiding} can be removed.

To prove that Breaker has a winning strategy for the \btatg with the given $q$, we extend Breaker's strategy for the \btg.
In addition to threat and closing edges, Breaker now also exhaustively claims two more types of edges: The first type consists of \emph{dangerous} edges that, if not claimed by Breaker, would allow Maker to create more than $q$ threats in the next turn and thereby lead to her immediate win. The second consists of \emph{high-degree} edges connecting high-degree vertices.

\begin{definition}
    A vertex $v$ is of \emph{high degree} (resp. \emph{low degree}) if $\degM(v)\geq q/2$ (resp. $\degM(v)<q/2$). An edge is \emph{high-degree} if both endpoints are high-degree. For an unclaimed edge $e$, denote by $\threats(e)$ the number of threats Maker would create by claiming it. 

    An edge $\{v,w\}$ is \emph{unclaimed dangerous} if it is unclaimed and $\threats(\{v,w\}) > q$. An edge $\{v,w\}$ is \emph{dangerous} if it was claimed by Breaker as an unclaimed dangerous edge and is not a high-degree edge.
\end{definition}
\cref{fig:dangerous} shows an example of a dangerous edge. With these definitions, we are now able to state Breaker's strategy.

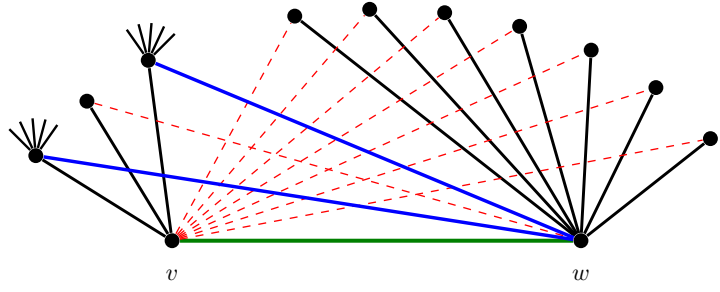
\begin{figure}
    \centering
    \begin{tikzpicture}[
        scale=0.9,
        transform shape,
        vertex/.style={
            circle,
            fill=black,
            inner sep=2.2pt
        },
        makeredge/.style={
            draw=black,
            line width=1.1pt
        },
        dangerousedge/.style={
            draw=green!50!black,
            line width=1.5pt
        },
        threatedge/.style={
            draw=red,
            dashed,
            line width=0.5pt
        },
        breakeredge/.style={
            draw=blue,
            line width=1.3pt
        },
        partial/.style={
            draw=black,
            line width=0.9pt
        }
    ]
        
        \node[vertex] (v) at (0,0) {};
        \node[vertex] (w) at (6,0) {};
        
        \node at (0,-0.5) {$v$};
        \node at (6,-0.5) {$w$};
        
        \draw[dangerousedge] (v) -- (w);

        
        \node[vertex] (a1) at (-2.0,1.25) {};
        \node[vertex] (a2) at (-1.25,2.05) {};
        \node[vertex] (a3) at (-0.35,2.65) {};
        
        \node[vertex] (b1) at (1.80,3.30) {};
        \node[vertex] (b2) at (2.90,3.40) {};
        \node[vertex] (b3) at (4.00,3.35) {};
        \node[vertex] (b4) at (5.10,3.15) {};
        \node[vertex] (b5) at (6.15,2.80) {};
        \node[vertex] (b6) at (7.10,2.25) {};
        \node[vertex] (b7) at (7.90,1.50) {};

        \draw[makeredge] (v) -- (a1);
        \draw[makeredge] (v) -- (a2);
        \draw[makeredge] (v) -- (a3);

        \foreach \x in {b1,b2,b3,b4,b5,b6,b7}
            \draw[makeredge] (w) -- (\x);

        
        \draw[threatedge] (a2) -- (w);
        
        \foreach \x in {b1,b2,b3,b4,b5,b6,b7}
            \draw[threatedge] (v) -- (\x);

        
        \foreach \angle in {55,75,95,115,135}
            \draw[partial] (a1) -- ++(\angle:0.55);
        
        \foreach \angle in {45,65,85,105,125}
            \draw[partial] (a3) -- ++(\angle:0.55);
        
        \draw[breakeredge] (a1) -- (w);
        \draw[breakeredge] (a3) -- (w);
        
    \end{tikzpicture}
    
    \caption{ Example of a dangerous unclaimed edge for $q=7$. Maker edges are marked black. High-degree Breaker edges are marked blue. Maker claiming the green edge $\{v,w\}$ would introduce $q+1$ threats (marked in red). Thus, $\{v,w\}$ is an unclaimed dangerous edge. High-degree neighbours of $v$ would not create threats with $\{v,w\}$, since $w$ is high-degree and Breaker claims high-degree edges.}
    \label{fig:dangerous}
\end{figure}

\begin{strategy}(Breaker's strategy for the \btatg)
    \label{strat:TA}
    At each turn, Breaker does the following:
    \begin{enumerate}
        \item\label{strat:TA claim all} If all remaining unclaimed edges can be claimed in this turn, claim all remaining unclaimed edges.
        \item\label{strat:TA threat} Claim all threat edges.
        \item\label{strat:TA high} Claim all unclaimed high-degree edges.
        \item\label{strat:TA closing} For every vertex $v$, if $\degM(v)=q$, claim all unclaimed edges incident to $v$.
        \item\label{strat:TA dangerous} Repeatedly claim one unclaimed dangerous edge until no unclaimed dangerous edges remain.
    \end{enumerate}
\end{strategy}
We call the process of claiming all edges incident to $v$ in \ref{strat:TA closing} \emph{closing} $v$, and the edges claimed during this step are called \emph{closing edges}. Vertices for which this process has been done are called \emph{closed}; all others are called \emph{open}. Note that the definition is slightly different from that in \Cref{sec: budget triangle game}. There, we only closed vertices if they reached a degree of $q+1$. Here, however, as $q\ge \sqrt{2n}$, we can afford to already close vertices when they reach a degree of~$q$.

One can quickly see that Breaker will never lose if he can follow this strategy, as claiming all dangerous edges prevents him from having to claim more than $q$ threat edges in the next round. Thus, we only need to show that if Breaker follows this strategy, he will never claim more than $q\cdot e(G_M)$ edges. This is clear if Breaker can perform step~\ref{strat:TA claim all} of \Cref{strat:TA} and claim all remaining edges. Hence, we can assume in the following that we are at a time of the game at which the condition of step~\ref{strat:TA claim all} is not satisfied yet.

To see that Breaker will always be able to perform \Cref{strat:TA}, we will proceed similarly to the last section and maintain a budget function that assigns the budget to the vertices, which is then used to pay for edges that are claimed by Breaker.

As we have seen in the last section, Breaker is able to pay for claimed threat edges and closing edges. High-degree edges can only occur when both endpoints are of high degree. Intuitively, this implies that Breaker has already accumulated enough budget at these vertices that he can use to pay for these edges. Hence, the crucial point of this analysis is to see how the dangerous edges can be paid for. 

We already remark here that the order in which Breaker claims edges is important: as we will see, it is vital that Breaker claims high-degree edges before determining and claiming all edges that are still dangerous (\cref{fig:dangerous}).
Note that whether an edge is dangerous can change during Breaker's turn. Crucially, in the strategy, Breaker only claims dangerous edges if they are still dangerous at the time of insertion, which allows us to bound the number of dangerous edges incident to any vertex (see \cref{lemma:dangerous n/2}).

\paragraph{Dangerous Edges.}

Since dangerous edges play such a crucial role, we start by observing a few basic properties of them.
\begin{observation}\label{obs:dangerousthreshold}
    Any unclaimed dangerous edge $e$ fulfils $\threats(e) = q+1$.  
\end{observation}
\begin{proof}
    Maker claiming an edge increases $\threats(e)$ by at most one for all edges $e$. However, as soon as $\threats(e) = q+1$ at step~\ref{strat:TA dangerous} of Breaker's turn, $e$ will be claimed by Breaker. Hence, $\threats(e)$ cannot increase beyond $q+1$.
\end{proof}

\begin{observation}
\label{obs:dangerous1lo1hi}
    Every dangerous edge $\{v,w\}$ has one low-degree and one high-degree endpoint.
\end{observation}
\begin{proof}
    Let us first consider the time when Breaker claims the dangerous edge $\{v,w\}$.
    Breaker claims dangerous edges in step~\ref{strat:TA dangerous} of his turn.
    If both endpoints were high-degree, the edge would have been previously claimed by Breaker as a high-degree edge in step~\ref{strat:TA high} of his turn. As Maker claiming $\{v,w\}$ can only create $\degM(v)+\degM(w)$ threats, an edge with two low-degree endpoints can never induce more than $q$ threats. Thus, $\{v,w\}$ has one low-degree and one high-degree endpoint at the time when Breaker claims $\{v,w\}$. 

    Later, this can only change if the low-degree endpoint becomes a high-degree endpoint, i.e.\@ both endpoints of $\{v,w\}$ are high-degree. But then, by definition, $\{v,w\}$ is no longer a dangerous edge but a high-degree edge.
\end{proof}
Since Breaker claims high-degree edges, these edges cannot count towards the threat threshold that defines an edge as (unclaimed) dangerous. This motivates the distinction between neighbours with a low degree and neighbours with a high degree. 
\begin{definition}
    For any vertex $v$, let $\degML(v)$ be the number of low-degree neighbours of $v$ in~$G_M$.
    For a high-degree vertex $u$, we define $\Nearly(u)$ as the set of all low-degree vertices which were among the first $q/2$ neighbours of $u$ in~$G_M$. Here, we use the fact that $q$ is even. Conversely, for a low-degree vertex $v$, we define $\Ninvearly(v)$ as the set of all high-degree vertices $u$ for which $v\in\Nearly(u)$.
\end{definition}
\begin{observation}
\label{obs:dangerousClaiming}
    At the time Breaker claims a dangerous edge $\{v,w\}$ with low-degree $v$, we have $\degML(v)+\degM(w) > q$. 
\end{observation}
\begin{proof}
    Let $u$ be a high-degree neighbour of $v$. Then, \Cref{obs:dangerous1lo1hi} implies that $\{u,w\}$ is a high-degree edge and was already claimed by Breaker. Hence, Maker claiming $\{v,w\}$ creates at most $\degML(v)+\degM(w)$ many threats.  
\end{proof}

\begin{observation}
    \label{obs:dangerousDecreasing}
    At any time of the game, if $\{v,w\}$ is a dangerous edge with low-degree $v$, we have $\degML(v)+\degM(w)+|\Ninvearly(v)|>q$. 
\end{observation}
\begin{proof}
    \Cref{obs:dangerousClaiming} implies that when $\{v,w\}$ is claimed by Breaker, it holds that $\degML(v)+\degM(w)+|\Ninvearly(v)| \geq \degML(v)+\degM(w)>q$.
    Over the course of the game, $\degM(w)$ and $\abs{\Ninvearly(v)}$ can only increase, but $\degML(v)$ might decrease if a low-degree neighbour $u$ of $v$ becomes high-degree. In this scenario, $\degML(v)$ decreases by 1, but $u$ is added to $\Ninvearly(v)$. Thus, the statement also holds at any later stage of the game.
\end{proof}

\paragraph{Vertex Budget Allocation and Payment.}

To prove \cref{thm: threatavoiding}, we need to show that Breaker has enough budget to play \cref{strat:TA}. As for the \btg, we will maintain a vertex budget function $\bud(v) \coleq\bud_A(v)-\bud_P(v)$, consisting of a positive term $\bud_A$ of allocated budget and a term $\bud_P$ of budget used for paying for Breaker's claimed edges.

As in the \btg, we set $\bud_A(v)\coleq\frac{q}{2}\degM(v)$, i.e., we allocate for each claimed Maker edge a budget of $q/2$ to each of its endpoints. Thus, if we now fix the game after Breaker's turn, the total budget distributed among all vertices equals the number of edges Breaker was allowed to claim until now: 
\begin{equation}
    \sum_{v\in V(G)} \bud_A(v)=q\cdot e(G_M). \label{eq:c_A}
\end{equation}

To define the payment term $\bud_P(v)$, we count the number of Breaker edges based on their type in the following way:
\begin{itemize}
    \item Let $\nT(v)$ be the number of (threat) edges incident to~$v$ claimed by Breaker in step \ref{strat:TA threat} of \Cref{strat:TA}.
    \item Let $\nH(v)$ be the number of high-degree edges claimed by Breaker incident to $v$ which were not claimed by Breaker in step \ref{strat:TA threat} of \Cref{strat:TA}.
    \item Let $\nD(v)$ be the number of dangerous edges claimed by Breaker incident to $v$ which were not claimed by Breaker in step \ref{strat:TA threat} of \Cref{strat:TA}.
    \item Let $\nC(v)$ be the number of closing edges claimed by Breaker when closing~$v$, which are not high-degree or dangerous.    
\end{itemize}
Recall that an edge can only be dangerous if it is not high-degree. Therefore, each edge of Breaker is counted by exactly one of the four functions defined above.

Any vertex will have enough budget to pay for half of incident threat edges. Similarly, if a vertex is high-degree, it will have accumulated enough budget to pay half for incident high-degree edges. These two edge types will therefore be paid for equally by their two endpoints. If a vertex is being closed, it has accumulated enough budget to pay the full amount for all its incident edges; therefore, closing edges are paid for by the closing vertex. 

Intuitively, it is reasonable to pay for dangerous edges by their endpoints, proportional to their respective degrees. However, it turns out that dangerous edges cannot be paid for only by their endpoints. Instead, for a dangerous edge $\{v,w\}$ with $v$ being low-degree, also the vertices in $\Ninvearly(v)$ will pay a certain amount proportional to their degree. Note that each vertex $u\in \Ninvearly(v)$ is high-degree, so $u$ has already accumulated a lot of budget, which is why $u$ can afford to partly pay for the dangerous edges incident to  $v\in \Nearly(u)$.
This motivates the following payment function. 
\begin{equation*}
    \bud_P(v) \coleq 
    \begin{cases}
      \frac{1}{2}\nT(v) + \frac{1}{2}\nH(v) + \frac{\degM(v)}{q}\nD(v) + \sum_{w\in \Nearly(v)}\frac{\alpha_v}{q}\nD(w) + \nC(v) & \degM(v) \ge \frac{q}{2}\\[4pt]
     \frac{1}{2}\nT(v)  + \frac{\degML(v)}{q}\nD(v) + \sum_{u\in \Ninvearly(v)}\frac{1-\alpha_u}{q}\nD(v) & \degM(v) < \frac{q}{2}
    \end{cases},
\end{equation*}
where $\alpha_{u}=\frac{2\degM(u)}{q}-1$ is a scalar with $\alpha_u\in[0,1]$ for each high-degree vertex $u$. One can think of the value $\alpha_u$ in the following way. Suppose a vertex $u$ was charged only $1/2$ for every incident Breaker edge. Once $u$ becomes high-degree, it holds that $\bud_A(u)=\frac{q}{2}\deg_M(u)=\frac{q}{2}\cdot \frac q2\ge n/2$, so $u$ has already accumulated enough budget to pay for any incident edge. Consequently, the budget $\frac{q}{2}$ received for each subsequent incident Maker edge is surplus, and sums to $\frac{q}{2}(\degM(u)-\frac{q}{2})$. This surplus is used by $u$ to pay for the dangerous edges incident to vertices in $\Nearly(u)$. Since $\abs{\Nearly(u)} \leq \frac{q}{2}$ and each vertex in $\Nearly(u)$ has at most $n$ incident dangerous edges, distributing the surplus evenly over these edges yields a per-edge payment of $\frac{q}{2}(\degM(u)-\frac{q}{2})/(\frac{q}{2}n) = (\degM(u)-\frac{q}{2})/n \ge (2\degM(u)-q)/q^2 = \frac{\alpha_u}{q}$, matching the definition of $\bud_P(v)$. It remains to show that $u$ can pay more than $1/2$ for each dangerous or closing edge incident to $u$ itself. This relies on two further properties: (i) there can be at most $\frac{n}{2}$ dangerous edges incident to each vertex in $\Nearly(u)$ (\Cref{lemma:dangerous n/2}), and (ii) since Breaker must close high-degree edges, the numbers of dangerous edges incident to vertices in $\Nearly(u)$, dangerous edges incident to $u$, and closing edges incident to $u$ cannot all be large simultaneously (\Cref{lem:fiveprops} \ref{eq:DSC<n}).

\begin{observation}\label{obs:edges are paid}
Throughout the game,
    \[\sum_{v\in V(G)} \bud_P(v) \geq e(G_B). \]
\end{observation}
\begin{proof}
    We prove that each Breaker edge contributes at least 1 to the sum. For each threat or high-degree edge $\{u,v\}$, both $\bud_P(u)$ and $\bud_P(v)$ are increased by $\frac{1}{2}$. When closing a vertex $v$, each closing edge contributes 1 to $\bud_P(v)$. Any dangerous edge $\{v,w\}$ has one low-degree endpoint, say $v$, and one high-degree endpoint, say $w$, by \cref{obs:dangerous1lo1hi}. Then, $\bud_P(w)$ contributes $\degM(w)/q$ to the sum, while $\bud_P(v)$ contributes $\degML(v)/q+\sum_{u\in \Ninvearly(v)}\frac{1-\alpha_u}{q}$. Additionally, each vertex $u\in \Ninvearly(v)$ contributes $\alpha_u/q$. In total
    \[\frac{\degM(w)}q + \frac{\degML(v)}q+\sum_{u\in \Ninvearly(v)}\frac{1-\alpha_u+\alpha_u}{q}=\frac{\degML(v)+\degM(w)+\abs{\Ninvearly(v)}}{q}\geq 1,\] by \cref{obs:dangerousDecreasing}.
\end{proof}

\begin{lemma}
\label{lemma:c_geq0}
    Throughout the game $\bud(v)\geq 0$ for all $v\in V(G)$.
\end{lemma}
Proving \Cref{lemma:c_geq0} allows us to prove the theorem.
\begin{proof}[Proof of \Cref{thm: threatavoiding}]
    We show that Breaker wins the game by playing \cref{strat:TA}. If Breaker is able to follow \cref{strat:TA} throughout the game, then he will always claim all threats and dangerous edges in his turn. Therefore, Maker will never claim a triangle, and by claiming all dangerous edges, Breaker prevents himself from having to claim more than $q$ threats in the next round. Once he has enough budget to perform step~\ref{strat:TA claim all} of his strategy, he will do it and win the game. It remains to show that, before he can perform step~\ref{strat:TA claim all} of his strategy, he always has enough budget to play according to the other steps of \cref{strat:TA}.
    
    Throughout the game, \eqref{eq:c_A}, \Cref{obs:edges are paid}, and \Cref{lemma:c_geq0} show that 
    \[0\leq \sum_{v\in V(G)} \bud(v) = \sum_{v\in V(G)} \bud_A(v) - \sum_{v\in V(G)} \bud_P(v)\leq q\cdot e(G_M) - e(G_B).\] 
    This implies $e(G_B) \leq q\cdot  e(G_M)$, showing that Breaker never claims more edges than his budget allows and thereby finishing the proof of \cref{thm: threatavoiding}.
\end{proof}
\medskip
To prove \cref{lemma:c_geq0}, we introduce the following notation: 
by $\nT_u(v)$ we denote the number of former threat edges $\{v,w\}$ incident to $v$ that were claimed by Breaker in step \ref{strat:TA threat} of \Cref{strat:TA} because $\big\{\{v,u\},\{w,u\}\big\}\subs E(G_M)$. In other words, $\nT_u(v)$ counts the number of former threat edges incident to $v$ that arose because of Maker edges incident to~$u$. Note that $\nT_u(v)=0$ unless $\{u,v\}\in E(G_M)$ and 
\[\nT(v)=\sum_{u\in N_{G_M}(v)} \nT_u(v).\]
Furthermore, for a high-degree vertex $w$, define 
\begin{equation*}
    \nS(w) \coloneqq \Big\vert \bigcup_{v\in \Nearly(w)}\big\{u\;:\;\{v,u\} \text{ dangerous}\big\}\Big\vert.
\end{equation*}
\Cref{fig:explanationSwTuv} illustrates these two definitions. Note that since $\degM(v) < \frac{q}{2}$ for all $v\in \Nearly(w)$, the vertices $u$ which are counted by $\nS(w)$ must fulfil $\degM(u) \geq \frac{q}{2}$ by \Cref{obs:dangerous1lo1hi}. 

\begin{figure}
    \centering

    \begin{subfigure}[b]{0.42\textwidth}
        \centering
        \begin{tikzpicture}[
            scale=0.7,
            vertex/.style={
                circle,
                fill=black,
                inner sep=2.2pt
            },
            blackedge/.style={
                draw=black,
                line width=1.1pt
            },
            rededge/.style={
                draw=red,
                line width=1.4pt
            }
        ]

        \node[vertex] (v) at (0,0) {};
        \node[vertex] (w) at (4,0) {};

        \node at (-0.05,-0.55) {$v$};
        \node at (4,-0.55) {$u$};

        \node[vertex] (x1) at (4.1,3.5) {};
        \node[vertex] (x2) at (5.0,3.15) {};
        \node[vertex] (x3) at (5.8,2.45) {};
        \node[vertex] (x4) at (6.5,1.65) {};

        \draw[blackedge] (v) -- (w);
        \draw[blackedge] (w) -- (x1);
        \draw[blackedge] (w) -- (x2);
        \draw[blackedge] (w) -- (x3);
        \draw[blackedge] (w) -- (x4);

        \draw[rededge]
            (v) to[out=35,in=225] (x1);

        \draw[rededge]
            (v) to[out=30,in=210] (x2);

        \draw[rededge]
            (v) to[out=25,in=200] (x3);

        \draw[rededge]
            (v) to[out=18,in=195] (x4);

        \end{tikzpicture}
        \caption{$t_u(v)$ former threats at $v$ via $u$.}
        \label{fig:first-part}
    \end{subfigure}
    \hfill
    \begin{subfigure}[b]{0.55\textwidth}
        \centering
        \begin{tikzpicture}[
            scale=0.7,
            vertex/.style={
                circle,
                fill=black,
                inner sep=2.2pt
            },
            neighbour/.style={
                circle,
                draw=black,
                fill=black,
                line width=1pt,
                inner sep=2.5pt
            },
            target/.style={
                rectangle,
                draw=black,
                fill=white,
                line width=1pt,
                minimum width=2mm,
                minimum height=2mm,
                inner sep=2.5pt
            },
            firstedge/.style={
                draw=black,
                line width=1.1pt
            },
            secondedge/.style={
                draw=green!50!black,
                line width=1.1pt
            }
        ]

        \node[vertex] (w) at (0,0) {};
        \node at (-0.15,-0.55) {$w$};

        \node[neighbour] (n1) at (4,3.6) {};
        \node[neighbour] (n2) at (4,2.6) {};
        \node[neighbour] (n3) at (4,1.6) {};
        \node[neighbour] (n4) at (4,0.6) {};

        \node[vertex] (u1) at (4,-1) {};
        \node[vertex] (u2) at (4,-2) {};

        \node[target] (s1) at (8,4.0) {};
        \node[target] (s2) at (8,3) {};
        \node[target] (s3) at (8,2) {};
        \node[target] (s4) at (8,1) {};
        \node[target] (s5) at (8,0) {};
        \node[target] (s6) at (8,-1) {};

        \draw[firstedge] (w) -- (n1);
        \draw[firstedge] (w) -- (n2);
        \draw[firstedge] (w) -- (n3);
        \draw[firstedge] (w) -- (n4);
        \draw[firstedge] (w) -- (u1);
        \draw[firstedge] (w) -- (u2);

        \draw[secondedge] (n1) -- (s1);
        \draw[secondedge] (n1) -- (s2);

        \draw[secondedge] (n2) -- (s2);

        \draw[secondedge] (n3) -- (s3);
        \draw[secondedge] (n3) -- (s4);
        \draw[secondedge] (n3) -- (s5);

        \draw[secondedge] (n4) -- (s6);

        \draw[
            dashed,
            line width=1pt
        ]
        (4,2.1) ellipse [x radius=0.65, y radius=2.25];

        \node at (2.8,4.0) {$N^{+}(w)$};

        \draw[
            dashed,
            line width=1pt
        ]
        (8,1.5) ellipse [x radius=0.75, y radius=3.25];

        \node at (9.25,4.0) {$\nS(w)$};

        \end{tikzpicture}
        \caption{Vertices counted in $s(w)$}
        \label{fig:second-part}
    \end{subfigure}
    \caption{Figures illustrating $\nT_u(v)$ and $\nS(w)$. Red edges are former threat edges claimed by Breaker; green edges are claimed dangerous edges; black edges are Maker's claimed edges. Left: The number of red edges equals $\nT_u(v)$. Right: Vertices shown as squares are the vertices counted in $\nS(w)$.}
    \label{fig:explanationSwTuv}
\end{figure}
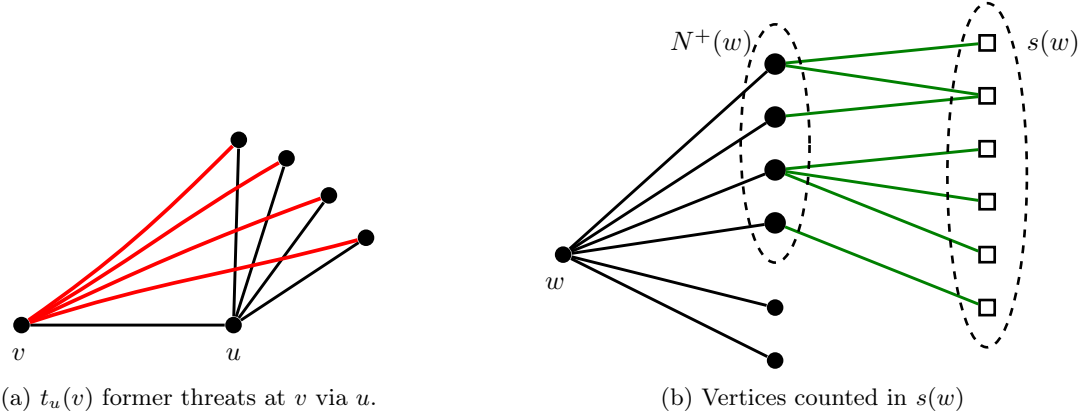

\begin{lemma}\label{lem:fiveprops}
    The following bounds hold throughout the game:
    \begin{enumerate}
        \item $\nT_u(v)\leq \degM(u) - 1$  for all $u,v\in V(G)$ for which $\{u,v\}\in E(G_M)$; \label{eq:T_u(v)}
        \item $\degM(v)\leq q$  for any $v\in V(G)$; \label{eq:degM{v}}
        \item $\nT(v) + \nH(v) + \nD(v) + \nC(v)\leq n$  for any $v\in V(G)$; \label{eq:THDC<n}
        \item $\nD(w)+\nS(w)+\nC(w)\leq n$ for all high-degree vertices $w\in V(G)$; \label{eq:DSC<n}
        \item $\nD(v)\leq \nS(w)$ for any high-degree vertex $w\in V(G)$ and all $v\in \Nearly(w)$.\label{eq:Dv<=Sw}
    \end{enumerate}
\end{lemma}
\begin{proof} We prove the properties one after another.
\begin{enumerate}
    \item For a fixed Maker edge $\{u,v\}$, only the Maker edges incident to $u$, except $\{u,v\}$, can induce a threat at $v$ counted in $\nT_u(v)$; see \Cref{fig:explanationSwTuv}.
    \item Since Breaker closes all vertices $v$ with degree $q$, thereby claiming all unclaimed edges incident to $v$, the property holds.
    \item This follows from the fact that $\nT(v)$, $\nH(v)$, $\nD(v)$, and $\nC(v)$ count distinct edges incident to $v$.
    \item Let $V_s$ be the vertex set counted by $\nS(w)$. Similarly, let $V_g$ be the set of vertices $u$ such that $\{u,w\}$ is counted by $\nD(w)$ and $V_c$ the set of vertices $u$ such that $\{u,w\}$ is counted by $\nC(w)$. We will show that $V_s$, $V_g$, and $V_c$ are pairwise disjoint.
    
    By \Cref{obs:dangerous1lo1hi}, any $v\in V_s$ is high-degree, while any $u\in V_g$ is low-degree. As $\nC(w)$ only counts closing edges which are not high-degree, any $u\in V_c$ must be low-degree as well. Thus, $V_s\cap V_g=\emptyset=V_s\cap V_c$. Finally, $\nC(w)$ only counts closing edges which are not dangerous. Thus, $V_c\cap V_g=\emptyset$. 
    \item This follows immediately from the definition.\qedhere
\end{enumerate}
\end{proof}
Next, we will use the fact that we iteratively claim dangerous edges in \cref{strat:TA}. This allows us to bound the total number of dangerous edges per vertex.
\begin{lemma}
\label{lemma:dangerous n/2}
    Throughout the game, $\nD(v)\leq n/2$ for all low-degree $v$.
\end{lemma}
\begin{proof} 
Fix a low-degree vertex~$v$. For any dangerous edge $e = \{v,x\}$, let $E^M_x$ be the subset of those Maker edges incident to $x$ that would have formed a threat together with $e$ if Maker had claimed $e$ instead of Breaker. See \Cref{fig:disjointEMx} for an illustration. Denote by $\degM'(v)$ the Maker degree of $v$ at the time $e$ was claimed. Since $v$ is low-degree, we also have $\degM'(v) \leq \frac{q}{2}$. Therefore, \Cref{obs:dangerousthreshold} implies that $\abs{E^M_x} \geq q+1 - \degM'(v) \geq \frac{q}{2}$. 

\begin{claim}
    Let $\{v,x\}$ and $\{v,x'\}$ be two different dangerous edges incident to $v$. Then $E^M_x \cap E^M_{x'} = \emptyset$.
\end{claim}
\begin{proof}
    Since edges in $E^M_x$ resp.\@ $E^M_{x'}$ must be incident to $x$ resp.\@ $x'$, it is enough to prove that $\{x,x'\}\not\in E^M_x \cap E^M_{x'}$. 
    Breaker claims dangerous edges iteratively; therefore, we may assume that $\{v,x\}$ was claimed before $\{v,x'\}$. In that case, $\{v,x\}$ cannot be a potential threat for $\{v,x'\}$, excluding $\{x,x'\}$ from $E^M_{x'}$.
\end{proof}

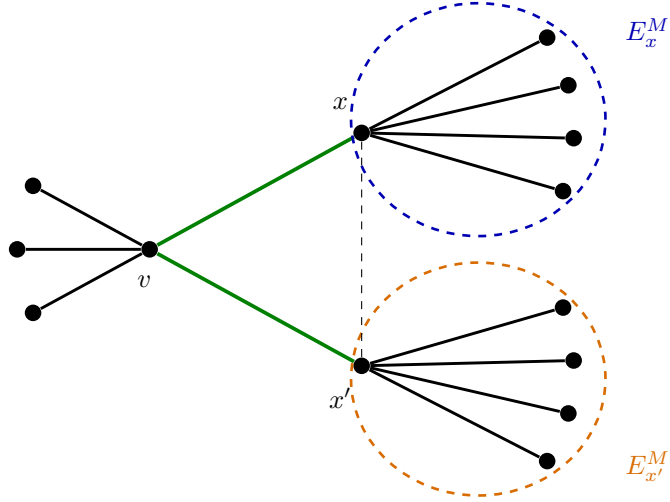
\begin{figure}
    \centering
    \begin{tikzpicture}[
        scale=0.7,
        vertex/.style={circle, fill=black, inner sep=2.2pt},
        makeredge/.style={draw=black, line width=1.1pt},
        dangerousedge/.style={draw=green!50!black, line width=1.4pt},
        setx/.style={draw=black, line width=1.1pt},
        setxp/.style={draw=black, line width=1.1pt}
    ]

    \node[vertex] (v)  at (0,0)    {};
    \node[vertex] (x)  at (4,2.2)  {};
    \node[vertex] (xp) at (4,-2.2) {};

    \node at (-0.1,-0.6) {$v$};
    \node at (3.6,2.75)  {$x$};
    \node at (3.6,-2.8)  {$x'$};

    \draw[dangerousedge] (v) -- (x);
    \draw[dangerousedge] (v) -- (xp);

    \node[vertex] (a1) at (-2.2,1.2)  {};
    \node[vertex] (a2) at (-2.5,0)    {};
    \node[vertex] (a3) at (-2.2,-1.2) {};
    \draw[makeredge] (v) -- (a1);
    \draw[makeredge] (v) -- (a2);
    \draw[makeredge] (v) -- (a3);

    \node[vertex] (y1) at (7.5,4.0) {};
    \node[vertex] (y2) at (7.9,3.1) {};
    \node[vertex] (y3) at (8.0,2.1) {};
    \node[vertex] (y4) at (7.8,1.1) {};
    \foreach \y in {y1,y2,y3,y4} \draw[setx] (x) -- (\y);

    \node[vertex] (z1) at (7.8,-1.1) {};
    \node[vertex] (z2) at (8.0,-2.1) {};
    \node[vertex] (z3) at (7.9,-3.1) {};
    \node[vertex] (z4) at (7.5,-4.0) {};
    \foreach \z in {z1,z2,z3,z4} \draw[setxp] (xp) -- (\z);

    \draw[dashed, line width=1pt, blue!70!black]
        (6.2,2.45) ellipse [x radius=2.4, y radius=2.2];
    \node[blue!70!black] at (9.4,4.1) {$E^M_x$};

    \draw[dashed, line width=1pt, orange!85!black]
        (6.2,-2.45) ellipse [x radius=2.4, y radius=2.2];
    \node[orange!85!black] at (9.4,-4.1) {$E^M_{x'}$};
    
    \draw[dashed] (x) -- (xp);

    \end{tikzpicture}
    \caption{Two dangerous edges $\{v,x\}$ and $\{v,x'\}$ at a low-degree vertex~$v$.
    Green edges are claimed dangerous edges; all other edges are Maker's claimed edges.
    Edges in the blue ellipse form $E^M_x$ and edges in the orange ellipse form $E^M_{x'}$.}
    \label{fig:disjointEMx}
\end{figure}

Thus, the total number of Maker edges is bounded by 
\[\frac{n^2}{2q}\geq e(G_M)\geq \abs{\bigcup_{\{v,x\}\text{ dangerous}} E^M_{x}}= \sum_{\{v,x\}\text{ dangerous}} \abs{E^M_{x}}\geq \sum_{\{v,x\}\text{ dangerous}} \frac{q}{2}=\frac{q}{2}\nD(v),\]
where the first inequality follows from \ref{strat:TA claim all} of \Cref{strat:TA} as the game ends after at most $\ceil{\frac{\binom n2}{q+1}}\le \frac{n^2}{2q}$ rounds.
This implies $\nD(v)\leq \frac{n^2}{q^2}\leq\frac{n}{2}$, since $q^2 \geq 2n$.
\end{proof}

We are now ready to give the missing proof for \Cref{lemma:c_geq0}.
\begin{proof}[Proof of \cref{lemma:c_geq0}]
    We do a case distinction on the degree of $v$.
    \begin{itemize}[label=--, left=0pt, itemindent=!, listparindent=\parindent, parsep=0pt]
        \item[] \emph{Case $\frac{q}{2}\leq \degM(v)\leq q$:} Let $x$ be such that $\degM(v) = \frac{q}{2}+x$.  If $x=0$, then
        \begin{align*}
           \bud(v)
            &=\left(\frac{q}{2}\right)^2 - \frac{1}{2}\nT(v) - \frac{1}{2}\nH(v) - \nC(v) - \frac{\frac{q}{2}}{q}\nD(v) - \sum_{w\in \Nearly(v)}\frac{\alpha_v}{q}\nD(w)\\
            &=\left(\frac{q}{2}\right)^2 - \frac{1}{2}\nT(v) - \frac{1}{2}\nH(v) - \frac{1}{2}\nD(v)\\
            &\ge \frac{1}{2}\big(n-\nT(v)-\nH(v)-\nD(v)\big)\\
            & \stackrel{\mathclap{\ref{eq:THDC<n}}}{\geq} 0,
        \end{align*} 
        where in the first inequality, we used $q^2\ge 2n$ and that $v$ is still open if $\degM(v)=\frac{q}{2}$.
        
        If $x>0$, note that
        \[\frac{\alpha_vq}{2x}=\frac{\left(\frac{2\degM(v)}{q}-1\right)q}{2x}=\frac{2\degM(v)-q}{2x}=\frac{q+2x-q}{2x}=1.\]
        Thus, we can conclude
        \begin{align*}
           \bud(v)
            &= \frac{q}{2}\left(\frac{q}{2}+x\right) - \frac{1}{2}\nT(v) - \frac{1}{2}\nH(v) - \nC(v) - \frac{\frac{q}{2}+x}{q}\nD(v) - \sum_{w\in \Nearly(v)}\frac{\alpha_v}{q}\nD(w)\\
            & \geq \frac{1}{2}\left(n - \nT(v) - \nH(v) - \nD(v)-\nC(v)\right) + \frac{x}{q}\left(n-\nC(v)-\nD(v) - \frac{\alpha_v}{x}\sum_{w\in \Nearly(v)}\nD(w)\right)\\
            & \stackrel{\mathclap{\ref{eq:THDC<n}}}{\geq} \frac{x}{q}\left(n-\nC(v)-\nD(v) - \frac{\alpha_v}{x}\sum_{w\in \Nearly(v)}\nD(w)\right)\\
            & \stackrel{\mathclap{\ref{eq:Dv<=Sw}}}{\geq} \frac{x}{q}\left(n -\nC(v) - \nD(v)-\frac{\alpha_v}{x}|\Nearly(v)|\nS(v)\right)\\[7pt]
            & \geq \frac{x}{q}\left(n - \nC(v) - \nD(v)-\frac{\alpha_vq}{2x}\nS(v)\right)\\[5pt]
            & = \frac{x}{q}\left(n - \nC(v) - \nD(v)-\nS(v)\right) \stackrel{\ref{eq:DSC<n}}{\geq} 0,
        \end{align*} 
        where in the first inequality we used $q^2\geq 2n$, and $\nC(v)=0$ unless $\frac{x}{q}\ge\frac{1}{2}$.

        \item[] \emph{Case $\degM(v)<  \frac{q}{2}$:}
        \begin{align*}
           \bud(v) 
            =&  \frac{q}{2}\degM(v) - \frac{1}{2}\nT(v)  - \frac{\degML(v)}{q}\nD(v) - \sum_{u\in \Ninvearly(v)}\frac{1-\alpha_u}{q}\nD(v)\\
            =& \sum_{\substack{w\in N_{G_M}(v)\\ \degM(w) < \frac{q}{2}}} \frac{q}{2} - \frac{\nT_w(v)}{2} - \frac{\nD(v)}{q}
            +\sum_{\substack{w\in N_{G_M}(v)\\ \degM(w) \geq \frac{q}{2}\\ w\not\in \Ninvearly(v)}} \frac{q}{2} - \frac{\nT_w(v)}{2}
            + \sum_{w\in \Ninvearly(v)} \frac{q}{2}- \frac{\nT_w(v)}{2} - \frac{1-\alpha_w}{q}\nD(v)\\
            &\hspace{-.87cm}\stackrel{\ref{eq:T_u(v)},\ref{eq:degM{v}},\ref{lemma:dangerous n/2}}{\geq} \sum_{\substack{w\in N_{G_M}(v)\\ \degM(w) < \frac{q}{2}}}  \frac{q}{2} - \frac{q}{4} - \frac{n}{2q} 
            +\sum_{\substack{w\in N_{G_M}(v)\\ \degM(w) \geq \frac{q}{2}\\ w\not\in \Ninvearly(v)}} \frac{q}{2} - \frac{q}{2} 
            + \sum_{w\in \Ninvearly(v)} \frac{q}{2}- \frac{\degM(w)}{2} - \frac{(1-\alpha_w)n}{2q}\\
            \geq & \sum_{w\in \Ninvearly(v)} \frac{q}{2}- \frac{\degM(w)}{2} - \frac{\left(1-\left(\frac{2\degM(w)}{q}-1\right)\right)n}{2q}\\
            = & \sum_{w\in \Ninvearly(v)} \frac{q}{2}- \frac{\degM(w)}{2} - \left(\frac{n}{q} - \frac{\degM(w)n}{q^2}\right)\\
            \geq & \sum_{w\in \Ninvearly(v)}\frac{q}{2}- \frac{\degM(w)}{2} - \left(\frac{q}{2} - \frac{\degM(w)}{2}\right)\geq 0.\qedhere
        \end{align*}
    \end{itemize}
\end{proof}

Now, we have proved that \Cref{strat:TA} is a winning strategy for the \btatg played on $K_n$ if $q$ is at least the smallest even integer satisfying $q\ge\ceil{\sqrt{2n}}$. Therefore, it is also a winning strategy for the \btg played on $K_n$ for those~$q$. Thus, we can ask again how big the budget usage is. As it turns out, it is still the worst possible amount asymptotically. 

\begin{proposition}
    Let $q$ be the smallest even integer satisfying $q\ge\ceil{\sqrt{2n}}$ and let $\mathcal S$ be \Cref{strat:TA} for the \btg played on $K_n$. Then $\budget(\mathcal S)=\Theta\big(n^2\big)$.
\end{proposition}
\begin{proof}
    As Breaker claims at most $\binom{n}{2}$ edges, we have $\budget(\mathcal S)\le \binom{n}{2}=\O\big(n^2\big)$.

    For the other direction, we give a strategy for Maker such that, if Breaker plays according to \Cref{strat:TA}, then his play has a budget usage of $\Omega\big(n^2\big)$. For this, let $d\coleq \frac{q}{2}-1$ and $k\coleq \floor{\frac{n}{4d}}$. Note that $d=\Theta\big(\sqrt n\big)$ and $k=\Theta\big(\sqrt n\big)$. Furthermore, let $G$ be the union of $k$ vertex-disjoint copies of $K_{d,d}$.
    
    In the first phase, Maker claims the edges of $G\subs K_n$. As all vertices have a degree smaller than $q/2$, all vertices are still low-degree, so Breaker does not claim any high-degree, closing, or dangerous edges. Furthermore, $e(G)$ rounds are clearly not enough rounds for Breaker to have enough budget to claim all remaining unclaimed edges. Therefore, Breaker might only claim some threat edges, but as $G$ is triangle-free, no edge of $G$ will be claimed by Breaker. After Maker claimed all edges of $G$, $G_M$ consists of $k$ copies of $K_{d,d}$ and no edges between the copies of $K_{d,d}$ have been claimed by Maker or Breaker. 

    In the next $\le 2dk$ rounds, while there is still a vertex $v\in V(G)$ which is low-degree and incident to an unclaimed edge, Maker will claim an edge incident to $v$. That way, $v$ becomes high-degree, and Breaker has to claim all edges from $v$ to any other high-degree vertices.

    After these $\le 2dk$ rounds, every $v\in V(G)$ is high-degree or not incident to an unclaimed edge. Therefore, all $(2d)^2\binom{k}{2}$ edges between the copies of $K_{d,d}$ have been claimed and Maker has claimed at most $2dk$ of them. Thus, Breaker has claimed at least
    \[(2d)^2\binom{k}2-2dk=\Theta\big(n^2\big)\]
    edges in $2dk=\Theta(n)$ rounds. As the budget Breaker got in these $2dk$ rounds is $2dk\cdot q=\Theta\big(n^{3/2}\big)$, this implies that the budget usage of $\mathcal S$ is $\Theta\big(n^2\big)$.
    \end{proof}

%% file: Chapter/4_kfour.tex
\section{The \texorpdfstring{\kfourg}{KFG}}\label{sec: K_4 game}
In this section, we prove the first upper bounds for the \kfourg and the \bkfourg (\Cref{thm: budget K_4} and \Cref{thm: non-budget K_4}). A lower bound for the \kfourg is given in \Cref{thm: kfour lower bound}. As this is done by a careful analysis of the existing proofs for the asymptotic bound given by \cite{hgameforhypergraphs}, we postpone the proof of \Cref{thm: kfour lower bound} to appendix \ref{sec: appendix lower bound kfourg}. \medskip

To prove an upper bound for the \bkfourg and the \kfourg, we state a winning strategy for Breaker. The winning strategy will be the same for both variants of the game. For this, we first have to generalise the threat definition to the $K_4$ setting.

\begin{definition}
    An unclaimed edge $\{u,v\}$ is a \emph{threat} if there are two distinct vertices $w,x\in V(G)\setminus\{u,v\}$ such that $\binom{\{u,v,w,x\}}{2}\setminus\big\{\{u,v\big\}\}\subs E(G_M)$.
\end{definition}

The idea behind Breaker's strategy is the following: if $\{u,v\}$ is a threat, then there must be two vertices $w,x\in V(G)$ such that $G_M$ contains both triangles $\{u,w,x\}$ and $\{v,w,x\}$ which share an edge $\{w,x\}$. Therefore, if Breaker can bound the number of triangles an edge is part of in $G_M$, he can also bound the number of threats Maker can create. Breaker will use one part of his edges to prevent these unions of triangles and the other part to close the threats that Maker is still able to create. 

To formalise this strategy, we first need the standard notion of a $t$-book.

\begin{definition}
    Given two vertices $v$ and $w$, a \textit{wedge over $v$ and $w$} consists of two edges of the form $\{v,u\},\{w,u\}$ where $u\in V(G)$. A \textit{$t$-book} consists of an edge $\{v,w\}$ and $t$ distinct wedges over $v$ and $w$, i.e., of $t$ triangles sharing an edge.
\end{definition}

Note that if there are no $t$-books in $G_M$, then every edge in $E(G_M)$ can be part of at most $t-1$ triangles in $G_M$. Our strategy for Breaker relies on the following result given by \citet{beck1982remarks} on a winning strategy in the \hypermakerbreaker, which we state here, adapted to our setting. See appendix \ref{sec: appendix} for a description of the \hypermakerbreaker and a transformation from an $H$-Game.

\begin{theorem}[\citet{beck1982remarks}, Theorem 1]\label{thm: beck}
    For a graph $H$ and a positive integer $n$, denote by $\mu_n(H)$ the number of copies of $H$ in the complete graph $K_n$. If $\mu_n(H)(1+q)^{-|E(H)|} < (1+q)^{-1}$, then Breaker has a winning strategy in the \hgame played on~$K_n$.
\end{theorem}

Now, we are ready to state Breaker's strategy for the \bkfourg and the \kfourg. The value of $t$ is set in the corresponding theorems later and depends on which game is played.

\begin{strategy}[Breaker's strategy for the \bkfourg and \kfourg]\label{strat: k4 budget}
    At each turn, Breaker does the following:
    \begin{enumerate}
        \item\label{enum: K4 book} Use the strategy given by \citet{beck1982remarks} to prevent Maker from claiming a $t$-book for some fixed $t$.
        \item\label{enum: K4 threat} Claim all threat edges. 
    \end{enumerate}
\end{strategy}

To prove that Breaker can apply this strategy, we will show that per turn Breaker will never claim more than $q_1$ edges to play part (i) of \cref{strat: k4 budget} and not more than $q_2$ edges to play part (ii) of \cref{strat: k4 budget} (amortised in the case of a budget). It then only remains to prove that $q_1 + q_2 \leq q$. The values for $q_1$ and $q_2$ will differ in the proofs depending on whether Breaker plays the \bkfourg or the \kfourg. Note that to play the strategy, it is not important to know the values of $q_1$ and $q_2$ but only $t$.\par

Let us start by giving an upper bound on the number of edges needed for part (i) of \cref{strat: k4 budget}.

\begin{lemma}\label{lem: bound q1 for strat (i)}
     Let $t=\omega(\log n)$ with $n$ running to infinity. If $n$ is large enough, part (i) of \Cref{strat: k4 budget} requires at most $\big(1+o(1)\big)\sqrt{ne/t}$ edges each round, without a budget.
\end{lemma}
\begin{proof}
    This is a direct consequence of \Cref{thm: beck}. Indeed, a $t$-book has $1+2t$ edges and in a complete graph there are $\binom{n}{2}\binom{n-2}{t}$ $t$-books. Following \Cref{thm: beck}, Breaker can prevent Maker from claiming a $t$-book in $K_n$ if the number $q$ of edges that Breaker may claim every turn fulfils
    \begin{align}
        &\binom{n}{2}\binom{n-2}{t}(q+1)^{-(1+2t)} < 1/(q+1)\notag\\
        \Leftrightarrow \quad & \left(\binom{n}{2}\binom{n-2}{t}\right)^{1/t} < (q + 1)^2.\label{eq: erdos selfridge q1 k4}
    \end{align}
    For $n$ large enough, it holds that
    \begin{equation*}
        \binom{n}{2}^{1/t} < 2^{-1/t}n^{2/t} = 2^{-1/t}e^{\frac{2\log n}{t}}  = 1+ o(1),
    \end{equation*}
    where we used $t= \omega(\log n)$ in the last equality. Furthermore, 
    \begin{align*}
        \binom{n-2}{t}^{1/t} < \binom{n}{t}^{1/t} < \left(\frac{ne}{t}\right)^{t/t} = \frac{ne}{t}.
    \end{align*}
    Therefore, if $n$ is large enough, the maximum number of edges needed to perform part (i) of \Cref{strat: k4 budget} is $\big(1+o(1)\big)\sqrt{ne/t}$.
\end{proof}
It is unknown if this result can be improved if Breaker were allowed to store edges in a budget. With this lemma in hand, we are finally able to prove the upper bounds for the \bkfourg and the \kfourg.

\thmbudgetkfour*
\begin{proof}
    Fix any $\eps>0$, let $n$ be large enough and set $q\coleq\ceil{\left(5\cdot 2^{-9/5}e^{2/5} + \eps \right)n^{2/5}}$.
    We will show that \Cref{strat: k4 budget} with $t \coleq \ceil{\left(2^{-2/5}e^{1/5}\right)n^{1/5}}$ is a winning strategy for Breaker for the \bkfourg played on  $K_n$.

    By \ref{enum: K4 threat} of \Cref{strat: k4 budget}, Maker will never be able to claim a~$K_4$. All we have to show is that Breaker can always perform the steps \ref{enum: K4 book} and \ref{enum: K4 threat} of \Cref{strat: k4 budget} without claiming more than $q\cdot  e(G_M)$ edges.
    
    Define $q_1 \coleq \ceil{(1+\eps/10)\sqrt{ne/t}} \le  \left(2^{1/5}e^{2/5}+\eps/2\right)n^{2/5}$ and $q_2 \coleq \ceil{t^2/2} \le \left(2^{-9/5}e^{2/5}+\eps/2\right)n^{2/5}$.
    Since $t \in\omega(\log n)$, \Cref{lem: bound q1 for strat (i)} implies that the number of edges Breaker needs to play part~\ref{enum: K4 book} of \Cref{strat: k4 budget} is bounded by~$q_1$. Next, we prove that a budget increasing by $q_2$ edges each turn is enough to play part~\ref{enum: K4 threat} of \Cref{strat: k4 budget}, i.e., to close all threats.
    \begin{claim}
        If the graph $G_M$ does not contain a $t$-book, then Maker can have created at most $|E(G_M)|\frac{t^2}{2}$ threats in total.
    \end{claim}
    \begin{proof}
        Fix some vertex $v$. First, we want to bound the number $e_v$ of edges incident to $v$ that were threat edges at some point. Assume $\{v,w\}$ has been a threat edge. Then there must be vertices $x$ and $y$ such that 
        \begin{equation}\label{eq: claim budget egm}
            \{v,x\},\{v,y\},\{w,x\},\{w,y\},\{x,y\}\in E(G_M).
        \end{equation}
        Thus, the number of vertex sets $\{w,x,y\}$ fulfilling \eqref{eq: claim budget egm} is an upper bound for $e_v$. Since $\{v,x\}\in E(G_M)$, there are at most $\degM(v)$ possibilities for $x$. Since $\{v,x\}$ can be in at most $t$ triangles in $G_M$, due to the assumption of the claim, there are at most $t$ choices for $y$ for a fixed $x$. The same argument applied to $\{x,y\}$ implies that there are at most $t$ choices for $w$ for fixed $x$ and $y$. Finally, $x$ and $y$ are interchangeable; hence  $e_v \leq \degM(v)\cdot \frac{t^2}{2}$.

        Summing over all vertices, we can bound the number of edges that became threats by 
        \begin{equation*}
            \frac{1}{2}\sum_{v\in V}e_v \le \frac{1}{2}\sum_{v\in V}\degM(v)\frac{t^2}{2} = |E(G_M)|\frac{t^2}{2}.\qedhere
        \end{equation*}
    \end{proof}
    
    Because $|E(G_M)|$ is equal to the number of Maker's turns, Breaker can play part~\ref{enum: K4 threat} of \Cref{strat: k4 budget} if he gets $q_2 \ge \frac{t^2}{2}$ edges every turn to claim or add to his budget.
    Since $q \geq q_1 + q_2$, Breaker wins playing \Cref{strat: k4 budget}.
\end{proof}

Note that if one could improve on the strategy given by \citet{beck1982remarks} by using a budget, the constant in \Cref{thm: budget K_4} might also change.\par  

Using the same strategy with another fixed $t$, we can give an upper bound on the threshold bias of the \kfourg. 

\thmnobudgetkfour*
\begin{proof}
    Fix any $\eps>0$, let $n$ be large enough and set $q\coleq \ceil{\left(5^{6/5}\;2^{-9/5}e^{2/5} + \eps \right)n^{2/5}}$. We will show that \Cref{strat: k4 budget} with $t\coleq\ceil{\left(10^{-2/5}e^{1/5}\right)n^{1/5}}$ is a winning strategy for Breaker for the \kfourg played on~$K_n$. 

    By \ref{enum: K4 threat} of \Cref{strat: k4 budget}, Maker will never be able to claim a~$K_4$. All we have to show is that Breaker can always perform the steps \ref{enum: K4 book} and \ref{enum: K4 threat} of \Cref{strat: k4 budget} without claiming more than $q$ edges in each turn.
    
    Define $q_1\coleq \ceil{(1+\eps/10)\sqrt{ne/t}} \le  \left(10^{1/5}\;e^{2/5}+\eps/2\right)n^{2/5}$ and $q_2 \coleq \ceil{\frac{5}{2}t^2} \le \left(5^{1/5}\;2^{-9/5}\;e^{2/5}+\eps/2\right)n^{2/5}$. The proof follows the same strategy as the proof of \Cref{thm: budget K_4}. Again, since $t=\omega(\log n)$, \Cref{lem: bound q1 for strat (i)} implies that the number of edges Breaker needs to play part~\ref{enum: K4 book} of \Cref{strat: k4 budget} is bounded by $q_1$. Next, we prove that Breaker does not claim more than $q_2$ edges each turn to play part~\ref{enum: K4 threat} of \Cref{strat: k4 budget}, i.e., to close all threats.

    \begin{claim}
        If the graph $G_M$ never contains a $t$-book, then Maker can create at most $\frac{5}{2}t^2\le q_2$ threats each turn.
    \end{claim}
    \begin{proof}
        All threats created in a turn must contain the new Maker edge $e = \{u,v\}$. We want to count how many possibilities there are to choose $w$ and $x$ such that $\{u,v,w,x\}$ is the vertex set of a threat. 
  
        First, assume that the missing edge for a $K_4$ is the edge opposite of $e$, i.e., the edge $\{w,x\}$. In this case, $\{u,v,w\}$ and $\{u,v,x\}$ form two triangles in~$G_M$. Due to the assumption of the claim and because of symmetry in $w$ and $x$, there are at most $\frac{t^2}{2}$ possibilities to choose $\{w,x\}$ in this case.  

        Now assume that the missing edge is incident to one of the endpoints of~$e$. At the cost of adding a factor of~2, we may assume that this endpoint is $v$; without loss of generality, let $\{v,x\}$ be the missing edge. Then $\{u,v,w\}$ and $\{u,w,x\}$ form triangles in~$G_M$. 
        Hence, there are at most $t$ choices for $w$, and, given $w$, at most $t$ choices for $x$. Therefore, there are at most $2t^2$ threats in this case.

        In total, Maker can create at most $\frac{1}{2}t^2 + 2t^2 = \frac{5}{2}t^2 \le q_2$ threats each turn.
    \end{proof}
    Hence, Breaker claims at most $q_2$ edges each round while following part (ii) of \Cref{strat: k4 budget}.
    Since $q \geq q_1 + q_2$, Breaker wins playing \Cref{strat: k4 budget}.
\end{proof}

%% file: Chapter/5_concluding.tex
\section{Concluding Remarks}\label{sec: Concluding}
In this paper, we showed that if Breaker is allowed to keep a budget, then the threshold bias of the \tg is $\big(1+o(1)\big)\sqrt{2n}$ and this is even true if he is not allowed to use the budget to close threats, i.e.\@ Breaker still loses if Maker creates more than $q$ threats in a single turn.

The big open question that still remains is whether the budget is necessary at all to get a threshold bias of $\big(1+o(1)\big)\sqrt{2n}$. To approach this, one can try to find winning strategies $\mathcal S$ in which the worst-case budget usage is smaller than that of \Cref{strat: Breaker triangle}.

\begin{problem}
    For $q=\qTG^b(n)$, is there a winning strategy $\mathcal S$ for Breaker in the \btg played on $K_n$ such that $\budget(S)=o\big(n^2\big)$?
\end{problem}

The definition of $\budget(\mathcal S)$ in \Cref{def: budget usage} was inspired by the norm $\norm\cdot_1$ for sequences. Of course, one could also look at other norms $\norm\cdot_p$. The larger the $p$, the more the large entries in the excess sequence are penalised. If one uses $\norm\cdot_\infty$, one tries to minimise the maximum entry of the excess sequence. Interestingly, in this setting,  \Cref{strat: Breaker triangle} is not of the order $\Theta(n^2)$ as both \ref{enum: threat} and \ref{enum: close} of \Cref{strat: Breaker triangle} only require Breaker to claim $\O(n)$ edges per turn.

Another way to bridge the gap between the budget and the non-budget version is as follows:

\begin{definition}
    For a fixed positive integer $n$, a real number $a\in [0,1]$, a family $\mathcal F\subs 2^{E(K_n)}$ of \emph{winning sets} and a non-negative integer $q$, the \emph{$(1:q)$ $a$-Budget Maker Breaker Game} is the following game with two players, called Maker and Breaker. The players take alternate turns, with Maker beginning and $B\coleq 0$ at the start. In her turn, Maker claims one of the previously unclaimed edges of $K_n$. In his turn, Breaker can claim up to $\floor{q+B}$ previously unclaimed edges of $K_n$. If $b\le q+B$ is the number of edges he claimed, $B$ is updated in the following way:
    \[B\coleq \begin{cases}
        B-(b-q)&b> q\\B+a(q-b)&b\le q.
    \end{cases}\]
    Maker wins if she manages to claim all elements of an $F\in\mathcal F$. Breaker wins if he manages to claim an element of each winning set. 
\end{definition}
In this game, $B$ keeps track of the budget Breaker currently has, and he can spend it at any time. However, when he claims fewer than $q$ in his turn, only an $a$ fraction of his unused edges is stored in the budget. The smaller $a$ is, the more incentive Breaker has to play his given $q$ edges per turn immediately. Note that $a=0$ yields the original $(1:q)$ Maker Breaker game, whereas $a=1$ yields the $(1:q)$ Budget Maker Breaker Game studied in this paper. Therefore, this version gives a continuous way to go from the budget to the non-budget version.

\begin{problem}
    Determine the infimum of all $a\in[0,1]$ for which the threshold bias of the \abtg is $\big(1+o(1)\big)\sqrt{2n}$.
\end{problem}

To get used to the $(1:q)$ $a$-Budget Maker Breaker Game, it might make sense to consider the \emph{$(1:q)$ $a$-Budget Hamilton Game} first, where Maker's goal is to claim the edges of a Hamilton cycle. \cite{Hamiltonresult} showed that in the non-budget version, i.e.\@ for $a=0$, the threshold bias is of the order $\Theta(n/\log n)$, whereas it is easy to see that the threshold bias for $a=1$ is just 1. It would be interesting to know how the threshold function behaves for values of $a$ that are between 0 and~1.

\begin{problem}
    Determine the threshold bias of the $a$-Budget Hamilton Game for all $a\in[0,1]$.
\end{problem}

%% file: Chapter/A_appendix.tex
\section{Appendix} \label{sec: appendix}

\subsection{Proof of the Lower Bound in the \texorpdfstring{\btg}{(1:q) Budget Triangle Game}}\label{sec: lower bound}
Here, we prove \Cref{lem: tg lower bound}, which uses the same strategy and arguments as \erdos{} and Chvátal used to prove the corresponding result in the \tg.

\tglowerbound*

To prove \Cref{lem: tg lower bound}, we have to describe a winning strategy for Maker if $q<\sqrt{2n-2}-\frac{3}{2}$. It turns out that the same strategy works as the one described by \cite{chvatal1978biased} when they initiated the study of the \tg, but we will analyse it a bit more thoroughly to get the above bound. We repeat the strategy here for completeness.

\begin{strategy}[Maker's strategy for the \tg, \cite{chvatal1978biased}]\label{strat: Maker triangle}
    Fix an arbitrary vertex $u$ throughout the game. At each turn, Maker does the following:
    \begin{enumerate}
        \item While there is still an unclaimed edge incident to $u$, claim it.
        \item Otherwise, find an unclaimed edge $vw$ such that $uv,uw\in E(G_M)$ and claim it.
    \end{enumerate}
\end{strategy}

The idea of the strategy is to build a large star so that eventually Breaker cannot claim all edges between endpoints of the star.

\begin{proof}[Proof of \Cref{lem: tg lower bound}]
We show that if $q<\sqrt{2n-2}-\frac{3}{2}$, Maker wins using \Cref{strat: Maker triangle}. To this end, we have to show that if all edges incident to $u$ are claimed for the first time, Maker can find the desired edge~$\{v,w\}$. Suppose this is not possible, and let $\ell\coleq\degM(u)$ at the time when all edges incident to $u$ are claimed.

Since all edges incident to $u$ are claimed and Maker cannot find the desired edge $\{v,w\}$, Breaker must have claimed at least $(n-1-\ell)+\binom{\ell}{2}$ edges. Therefore, we must have
\begin{align*}
   &&(n-1-\ell)+\binom{\ell}{2}&\le \ell\cdot q\\ 
   &\iff&\frac{1}{2}\ell^2-\!\left(q+\frac{3}{2}\right)\ell+(n-1)&\le 0\\
   &\iff& \ell^2-(2q+3)\ell+(2n-2)&\le 0.
\end{align*}
We can conclude 
\begin{equation}\label{eq: ell upper and lower bound}
    q+\frac{3}{2}-\sqrt{\left(q+\frac 32\right)^2-2n+2}\le \ell\le q+\frac{3}{2}+\sqrt{\left(q+\frac 32\right)^2-2n+2}.
\end{equation}
But since $q<-\frac{3}{2}+\sqrt{2n-2}$, we get
\[\left(q+\frac 32\right)^2-2n+2<0,\]
so the radicands in \eqref{eq: ell upper and lower bound} are negative, which is a contradiction.
\end{proof}

\subsection{Asymptotic bound of the \texorpdfstring{\bhgame}{(1:q) Budget H-Game}} \label{sec: appendix asymp bound}
Here, we prove that the asymptotic bound for the threshold bias of the \hgame also holds for the \bhgame.

\Hresult*

\citet{bednarska2000biased} showed this statement for the threshold bias of the original \hgame, i.e., they proved $q_H = \Theta\big(n^{1/m(H)})$.  Since $q_{\HG}^b(n)\le q_{\HG}(n)$, we only have to show that $q_{\HG}^b(n)=\Omega\big(n^{1/m(H)}\big)$. For this, we follow the same random strategy that is used by \cite{bednarska2000biased} to show that $q_{\HG}(n)=\Omega\big(n^{1/m(H)}\big)$. 

\begin{strategy}(Random Maker-Strategy, \cite{bednarska2000biased})\label{strat: random maker}
Maker always selects the next edge uniformly at random from all edges she has not sampled yet. If, by any chance, she selects an edge that has already been claimed by Breaker, her turn is marked as a \emph{failure} and she does not claim an edge in that round.
\end{strategy}

Following this strategy, after $m$ rounds, $G_M$ will be an instance of the random graph $\bG(n,m)$ where some edges are marked as failures. The core of the analysis is the following lemma, which says that even if a small fraction $\delta m$ of the edges in $\bG(n,m)$ are marked as failures, there will be a copy of $H$ in the remaining $(1-\delta)m$ edges of $\bG(n,m)$.

\begin{lemma}[\cite{bednarska2000biased}, Lemma 4]\label{lem: Random graph fact}
    For every graph $H$ containing at least one cycle, there exist constants $0< \delta<1$ and $n_0$ such that for $n\ge n_0$ and $m=2\ceil{n^{2-1/m(H)}}$ with probability at least 2/3, each subgraph of $\bG(n,m)$ with $\floor{(1-\delta)m}$ edges contains a copy of $H$.
\end{lemma}

The same strategy also works for the budget version, because if Breaker does not claim all possible edges in his turn and instead builds up a budget, this only decreases the probability that a turn by Maker is a failure. 

\begin{proof}[Proof of \Cref{prop: H-game thresholds} (analogous to \cite{bednarska2000biased}).]
It suffices to show that $q_{\HG}^b(n)=\Omega\big(n^{1/m(H)}\big)$.

First, we consider the case where $H$ is acyclic. If $H$ is a matching, then, as $H$ has at least two edges, $m(H)=1/2$ and Maker can win greedily in $e(H)$ moves as long as $q\le\binom{n}{2}/\big(e(H)-1\big)-2n$. If $H$ contains a path with three vertices, then $m(H)=1$ and Maker can win greedily in $e(H)$ moves if $q\le \big(n-2e(H)\big)/\big(e(H)-1\big)$. Thus, we can assume in the following that $H$ contains a cycle.

Let $0<\delta<1$ and $n_0$ be given as in \Cref{lem: Random graph fact} and assume that $q=0.1\cdot \delta n^{1/m(H)}$ and $n> n_0$. Let Maker follow \Cref{strat: random maker} and consider the game after 
\[m\coleq 2\ceil{n^{2-1/m(H)}}\le \frac{\delta}{2(q+1)}\binom n2\]
rounds.
Now, $G_M$ is an instance of $\bG(n,m)$ where some edges may be marked as failures. At each of these $m$ rounds, at most a proportion of $\delta/2$ of the edges have been claimed. Thus, at each turn of Maker, the probability that this turn will be a failure is at most $\delta/2$. Thus, the expected number of rounds that were failures is at most $\frac\delta 2m$. Therefore, Markov's inequality implies that the probability that at most $\delta m$ rounds were failures is at least $1/2$. By \Cref{lem: Random graph fact}, the probability that Maker has claimed the edges of a copy of $H$ is at least $1/6$. Thus, $q_{\HG}^b(n)=\Omega(n^{1/m(H)}\big)$.
\end{proof}

\subsection{Lower bound for the \texorpdfstring{\kfourg}{(1:q) K4-Game}}\label{sec: appendix lower bound kfourg}
Our goal is to prove the following lower bound for the \kfourg which was stated in \Cref{sec: introduction}. 

\kfourlowerbound*

It follows from a careful analysis of the proof of the asymptotic bound given by \cite{hgameforhypergraphs}. The underlying statement is stated in the context of a \hypermakerbreaker, so let us start by defining this type of game. A \emph{\hypermakerbreaker} is played by two players on a hypergraph $\CH$, in the following way: in each round, Maker first claims one of the previously unclaimed vertices in $V(\CH)$, whereafter Breaker claims $q$ of the previously unclaimed vertices of $V(\CH)$. Maker wins the game if, at some point, she has claimed all vertices of some hyperedge $e\in E(\CH)$. \hypermakerbreakers generalise \hgames. Indeed, given some graph $H$, we can model the \hgame by creating a hypergraph $\CH$ such that vertices of $\CH$ correspond to edges of $K_n$ and hyperedges of $\CH$ correspond to edge sets in $K_n$ that form a copy of $H$ in $K_n$. This implies $|V(\CH)| = |E(K_n)|$ and $|e| = |E(H)|$ for $e \in E(\CH)$. 

Before we can state the theorem from \cite{hgameforhypergraphs} that we rely on, we need to introduce some notation used by \cite{hgameforhypergraphs}. 

\begin{definition}
    For some hypergraph $\CH$, denote by $v(\CH) \coloneqq |V(\CH)|$ the number of vertices of $\CH$ and by $e(\CH)$ the number of edges of $\CH$. Furthermore, denote by $d(\CH) = e(\CH)/v(\CH)$ the density of $\CH$. 
    
    For a subset $S\subs V(\CH)$ we define $\deg(S) \coloneqq \vert\{e\in \CH: S\subs e\}\vert$ and set $\Delta_\ell(\CH) \coloneqq \max \{\deg(S):S\subs V(\CH), |S| = \ell\}$ to be the maximum $\ell$-degree of $\CH$ for any integer $\ell$. 

    Finally, if $\CH$ is $k$-uniform, we define $f(\CH) \coloneqq \min_{2\leq \ell \leq k}\left(\frac{d(\CH)}{\Delta_\ell(\CH)}\right)^{1/(\ell-1)}$.
\end{definition}

Now, we are ready to state the theorem from \cite{hgameforhypergraphs} that we will use to prove \Cref{thm: kfour lower bound}.

\begin{theorem}[\citet{hgameforhypergraphs}, Theorem 2.1]\label{thm: kusch 2.1}
    For every $k\geq 2$ and every positive $c_1\geq k$, there exist $c = c(k,c_1)> 0$ and $\overline{c} = \overline{c}(k,c_1)>0$ such that the following holds. If $\CH$ is a $k$-uniform hypergraph satisfying 
    \begin{align*}
        &\Delta_1(\CH)\leq c_1 d(\CH)\\[5pt]
        &f(\CH) > 1,\\[5pt]
        &\frac{v(\CH)}{f(\CH)}\left(1-\frac{1}{f(\CH)}\right) \geq \overline{c},
    \end{align*}
    then Maker has a winning strategy in the {\upshape\hypermakerbreaker} on $\CH$ provided {$q\leq cf(\CH)-1$}.
\end{theorem}

We want to state $c$ explicitly. In Section 2.1 of \cite{hgameforhypergraphs}, $c$ is defined to be $c = (1-\delta)/4> 0$ for some carefully chosen $\delta < 1$. After the statement of Lemma 2.5 in \cite{hgameforhypergraphs}, they set $\delta \in (1/2,1)$ to be such that $(1-\delta)(1-\ln(1-\delta)) < c'/4$ for some fixed $c'$. Finally, shortly before stating Lemma 2.5 in \cite{hgameforhypergraphs}, they set $c' = 1/(c_12^k)$. Hence, \ref{thm: kusch 2.1} is true for any $c$ that equals $\frac{x}{4}$ for some $x \in (0,1/2)$ fulfilling $x(1-\ln x)< (c_12^{k+2})^{-1}$.\bigskip

Next, we want to rewrite the statement in the context of \hgames. First, we adapt the notation. 

\begin{definition}
    For some fixed graph $H$ and a fixed $n$, let $\mu(H,n)$ denote the number of copies of $H$ in~$K_n$. Set $d(H,n) = \frac{2\mu(H,n)}{n(n-1)}$. Furthermore, for some $S\subs E(K_n)$ let $\mu(H,n,S)$ denote the number of copies of $H$ in $K_n$ that contain $S$. With this, we can define $\Delta_\ell(H,n) \coloneqq \max\{\mu(H,n,S): S\subs E(K_n), |S| = \ell\}$ and finally 
    \begin{equation*}
        f(H,n) \coloneqq \min_{2\leq \ell \leq |E(H)|} \left(\frac{d(H,n)}{\Delta_\ell(H,n)}\right)^{1/(\ell-1)}.
    \end{equation*} 
\end{definition}
It is easy to verify that this notation is compatible with our reduction of a \hgame to a \hypermakerbreaker. For this, also note that the constructed hypergraph $\CH$ is $|E(H)|$-uniform. 

Furthermore, note that $\overline{c}$ in the statement of \Cref{thm: kusch 2.1} only depends on $k = |E(H)|$ and $c_1$ and therefore, in our reduction, does not depend on $n$. All of this allows us to rewrite the theorem in the context of \hgames in the following way.

\begin{theorem}\label{thm: kusch with H}
    Fix $H$, $c \geq |E(H)|$ and $N\in \mathbb{N}$ and assume they satisfy 
    \begin{align}
        &\Delta_1(H,n) \leq c\cdot d(H,n) & \text{for all } n\geq N\label{eq: thm kusch 1}\\[5pt]
        &f(H,n) > 1 & \text{for all } n\geq N\label{eq: thm kusch 2}\\[5pt]
        &\frac{n^2}{f(H,n)}\left(1-\frac{1}{f(H,n)}\right) \stackrel{n\to\infty}{\to}\infty.\label{eq: thm kusch 3}
    \end{align}
    Fix $0<x<\frac{1}{2}$ which satisfies $x(1-\ln x) < \left(c\cdot 2^{|E(H)|+2}\right)^{-1}$. Then Maker has a winning strategy in the \hgame provided $q \leq \frac{x}{4} f(H,n) - 1$ and $n$ is large enough.
\end{theorem}

Now, we are ready to prove \Cref{thm: kfour lower bound}. 

\begin{proof}[Proof of \Cref{thm: kfour lower bound}.]
    We apply \Cref{thm: kusch with H} with $H = K_4$, $c = 6$ and $x = 5\cdot10^{-5}$. It remains to verify the conditions in \Cref{thm: kusch with H} and that it gives the desired bound. 

    It holds that $\mu(K_4,n) = \binom{n}{4}$ and therefore $d(H,n)=\frac{\binom{n}{4}}{\binom{n}{2}} = \frac{(n-2)(n-3)}{12}$. Every fixed edge $\{x,y\}\in E(K_n)$ is contained in a $K_4$ by choosing two further vertices; therefore $\Delta_1(K_4,n) = \binom{n-2}{2}$. Let $S\coleq \{e,f\}\subs E(K_n)$ be an edge set of size~2. If $e$ and $f$ are not adjacent, then $\mu(K_4,n,S) = 1$; otherwise $\mu(K_4,n,S) = n-3$ since the missing vertex of a $K_4$ can be chosen in $n-3$ ways. Hence, $\Delta_2(K_4,n) = n-3$. Similarly $\Delta_3(K_4,n) = n-3$. Here, the maximum occurs when the three edges in $S$ form a triangle; then the fourth vertex of the $K_4$ can again be chosen in $n-3$ ways.
    
    Finally,
    \begin{equation*}
        \Delta_4(K_4,n) = \Delta_5(K_4,n) = \Delta_6(K_4,n) = 1.
    \end{equation*}
     Indeed, any set of at least four edges, if contained in a $K_4$ at all, determines that $K_4$ uniquely.

     This allows us to verify Condition (\ref{eq: thm kusch 1}). Indeed, 
     \begin{equation*}
         \Delta_1(K_4,n) = \binom{n-2}{2} = \frac{(n-2)(n-3)}{2} = 6\frac{(n-2)(n-3)}{12} = c\cdot d(K_4,n).
     \end{equation*}

     Next, we compute
     \begin{equation*}
         f(K_4,n) = \min\left\{\frac{d(K_4,n)}{n-3}, \left(\frac{d(K_4,n)}{n-3}\right)^{1/2}, d(K_4,n)^{1/3}, d(K_4,n)^{1/4}, d(K_4,n)^{1/5}
        \right\}.
     \end{equation*}
     Since $d(K_4,n) = \Theta(n^2)$, the minimum is attained, for all sufficiently large $n$, by the last term. Hence,
     \begin{equation*}
         f(K_4,n) = \left(\frac{(n-2)(n-3)}{12}\right)^{1/5} \geq 32^{-1/5}n^{2/5} = \frac{1}{2}n^{2/5},
     \end{equation*}
     for $n$ large enough. Therefore $f(H,n) \to \infty$ as $n$ increases and the same holds for
     \begin{equation*}
         \frac{n^2}{f(H,n)}\left(1-\frac{1}{f(H,n)}\right) = \Theta(n^{8/5}).
     \end{equation*}
     Therefore, Conditions (\ref{eq: thm kusch 2}) and (\ref{eq: thm kusch 3}) are also fulfilled for large enough $n$. Finally, for $x = 5\cdot10^{-5}$, 
     \begin{equation*}
         x(1-\ln x) < 5.5\cdot 10^{-4} < 6.5\cdot 10^{-4} < \frac{1}{1536} = \frac{1}{6\cdot 2^{8}} = \frac{1}{c\cdot 2^{|E(K_4)|+2}}. 
     \end{equation*}
     Hence, for $n$ large enough and $q\leq 6\cdot 10^{-6}n^{2/5}$, we have 
     \begin{equation*}
         q < \frac{5}{8}\cdot 10^{-5}n^{2/5} -1 \leq \frac{x}{4}f(H,n)-1.
     \end{equation*}
     Thus, by \cref{thm: kusch with H}, Maker has a winning strategy in the \kfourg for $q\leq 6\cdot 10^{-6}n^{2/5}$.
\end{proof}